\documentclass[12pt]{article}
\usepackage{amssymb}
\usepackage{color}
\usepackage{hyperref}
\usepackage{amsmath}
\usepackage{amsfonts}
\usepackage{graphicx}
\usepackage{amsthm}
\usepackage{float}
\usepackage{xspace}
\usepackage{stmaryrd}

\newcommand{\RR}{\mathbb{R}}

\newcommand{\x}[1]{\operatorname{#1}}
\newcommand{\dom}{\x{D}}
\newcommand{\spec}{\x{Spec}}
\renewcommand{\Re}{\x{Re}}
\renewcommand{\Im}{\x{Im}}
\renewcommand{\b}[1]{\underline{#1}}
\renewcommand{\frak}[1]{\mathfrak{#1}}
\renewcommand{\ker}{\operatorname{Ker}}
\newcommand{\dx}{\mathrm{d}x}
\newcommand{\pp}{p}
\newcommand{\qq}{q}

\renewcommand{\geq}{\geqslant}
\renewcommand{\leq}{\leqslant}

\newtheorem{thm}{Theorem}[section]
\newtheorem{lem}[thm]{Lemma}

\newtheorem{col}[thm]{Corollary}

\title{On the convergence of the quadratic method}
\author{Lyonell Boulton \and Aatef Hobiny}
\date{4th August 2014}

\begin{document}

\maketitle

\abstract{The convergence of the so-called quadratic method for computing eigenvalue enclosures of general self-adjoint operators is examined. Explicit asymptotic bounds for convergence to isolated eigenvalues are found. These bounds turn out to improve significantly upon those determined in previous investigations. The theory is illustrated by means of several numerical experiments performed on particularly simple benchmark models of one-dimensional Schr{\"o}dinger operators.}

\tableofcontents

\newpage

\section{Introduction}
The Galerkin method is widely regarded as one of the best techniques for determining numerical one-side bounds for the eigenvalues of semi-definite operators. Under fairly general conditions, it leads to what is often called optimal order of convergence \cite{1983Chatelin}. The computation of complementary bounds for eigenvalues in the context of the projection methods is a more subtle task. Although various techniques are currently available, and historically they have existed for a long time (see e.g. \cite{1949Kato} and \cite[Chapter~4]{1974Weinberger}), only a few of them are as robust as the Galerkin method. The quadratic method, which relies on calculation of the  second order (relative) spectrum, is one of the few methods in this group which is capable of providing certified \emph{a priori} intervals of spectral enclosure. 

 Second order relative spectra were first considered by Davies \cite{1998Davies} in the context of resonances for general self-adjoint operators. It was then suggested by Shargorodsky \cite{Shargorodsky} and subsequently by Levitin and Shargorodsky \cite{Levitin}, that second order spectra could also be employed for the pollution-free computation of eigenvalues in gaps of  the essential spectrum. Based on these observations, convergence and spectral exactness was subsequently examined in \cite{200666Boulton, 20077Boulton, 20066Boulton, 2010Boulton}.  
 
 Various implementations, including on models from elasticity \cite{Levitin}, solid state physics \cite{2006Boulton}, relativistic quantum mechanics \cite{2008Boulton} and magnetohydrodynamics \cite{2010strauss}, confirm that the quadratic method is a reliable tool for eigenvalue approximation in the spectral pollution regime.  The goal of this paper is to continue with this programme of examining the general properties of the quadratic method and its potential use in the Mathematical Physics.

Section~\ref{quasec} is devoted to the basic setting of second order spectra, and how to determine from them upper and lower bounds for eigenvalues. Theorem~\ref{thmshar} below includes a short proof of Shargorodsky's Corollary \cite[Corollary~3.4]{Shargorodsky} in the general unbounded setting. 

In Section \ref{convergencesec} we address the spectral exactness of the quadratic method for general self-adjoint operators.  Our main contribution is Corollary~\ref{convquadgeneral}, which improves upon \cite[Theorem~3.4]{2010Boulton} in two crucial aspects. We consider a weaker hypothesis which includes approximation in the form sense rather than in the operator sense. This broadens the scope of applicability of the exactness result and it allows a sharpening of the precise exponent in the ratio of convergence of second order spectra to the spectrum (see Corollary~\ref{cor62}).

The theoretical framework of sections~\ref{quasec} and \ref{convergencesec} is completely general in character. In Section~\ref{schrosec} however, we have chosen applications to semi-bounded Schr{\"o}dinger operators in one dimension with potential singular at infinity. This allows us to illustrate our findings on the simplest possible model. Moreover, it highlights the effective use of the quadratic method for computation of complementary bounds for eigenvalues of semi-definite operators with compact resolvent. The latter is certainly a new possible application of the method which might be worth exploring in further detail.

 Section~\ref{numerical} contains specific numerical calculations on two exptremely well-known models: the harmonic and the anharmonic oscillators.  For these experiments the trial spaces are constructed via Hermite finite elements of order 3 and exact integrations. In an Appendix, we include all the explicit expressions involved in the assembly of the mass, stiffness and bending matrices, needed for computation of the second order spectra.

\subsubsection*{Notation}

Below $\mathcal{H}$ denotes a generic separable Hilbert space with inner product $\langle \cdot, \cdot \rangle$ and norm $\left\|\cdot\right\|$. Let the operator $A:\dom (A) \longrightarrow \mathcal{H}$ be self-adjoint. We will write $\spec(A)$ to denote the spectrum of $A$. 

For $u,\,v\in \dom(A)$ and $t\in\mathbb{R}$, let
\begin{gather*}
     \frak{a}^0_t(u,v)=\langle u,v \rangle \qquad \frak{a}^1_t(u,v)=\langle (A-t) u,v \rangle \\
      \frak{a}^2_t(u,v)=\langle (A-t)u,(A-t)v \rangle.
\end{gather*}
Whenever $t=0$ we will suppress the sub-index and write
\[\frak{a}^j(u,v)=\frak{a}^j_0(u,v).\] 

Given subspaces $\mathcal{L} \subset \dom(A) $ of dimension $n$ such that\footnote{Here and elsewhere below $\{b_j\}^n_{j=1}$ is a basis for $\mathcal{L}$.}
\[
\mathcal{L}=\x{Span} \{b_j\}^n_{j=1},
\]
 we will write 
\[
\mathbf{A}_{l}=\left[ \frak{a}^l (b_j,b_k)\right]^n_{jk=1} \in \mathbb{C}^{n \times n}.
\]
Without further mention, here and below we identify $v \in \mathcal{L}$  with $\b{v} \in \mathbb{C}^n$ 
by means of 
\[
  v=\sum_{j=1}^n \hat{v}(j)b_j \qquad \iff \qquad  \b{v}=\begin{bmatrix}\hat{v}(1)\\ \vdots \\ \hat{v}(n)\end{bmatrix}.
\]


\section{The quadratic method} \label{quasec}

We begin by describing the basic framework of the so called second order relative spectrum associated to
a self-adjoint operator, first considered by Davies in \cite{1998Davies}. We provide a short proof of Shargorodsky's Corollary \cite[Corollary~3.4]{Shargorodsky} in the unbounded setting, see \cite{Levitin} and Theorem~\ref{thmshar} below. We then re-examine mapping properties of the 
second order spectrum, as described in \cite{2010Boulton}. The proof of most statements in this section can be found scattered around the references \cite{1998Davies,Shargorodsky,Levitin,2010Boulton}. For the benefit of the reader, we include here a self-contained presentation.

The \emph{second order spectrum of $A$ (relative to the subspace $\mathcal{L}$)} is, by definition, the set
\[
\spec_2(A,\mathcal{L})\!=\!\{z\in \mathbb{C} :\exists u\in \mathcal{L}\setminus \{0\},\langle(A-zI)u,(A-\bar{z}I)v\rangle=0\ \forall v\in\mathcal{L}\}.
\]
That is, a complex number $z \in \spec_2(A,\mathcal{L})$, if and only if there exists a non-zero $u \in \mathcal{L}$ such that 
\begin{equation} \label{eq21}
 \frak{a}^2(u,v) -2z \frak{a}^1(u,v) +z^2 \frak{a}^0(u,v)=0 \qquad \forall v \in \mathcal{L}.
\end{equation}

For a basis  $\{b_j\}_{j=1}^n$ of $\mathcal{L}$ not necessarily orthonormal,  the weakly formulated problem \eqref{eq21} can be solved in matrix form via the quadratic matrix polynomial
\begin{equation} \label{eq22}
 Q(z)=\mathbf{A}_2-2z \mathbf{A}_1 +z^2 \mathbf{A}_0.
\end{equation}
Indeed $z\in \spec_2(A,\mathcal{L})$ if and only if $0=\det Q(z)$. 
The latter is a polynomial in $z$ of order $2n$. Since $\mathbf{A}_l$ are all hermitean,  $Q(z)^*=Q(\overline{z})$. Hence $\det Q(z)$ has all its coefficients real. Thus $\spec_2(A,\mathcal{L})$ is a set comprising at most $n$ different conjugate pairs\footnote{This includes the possibility of real numbers.}. See for example the figures~\ref{fig5_1} and \ref{fig5_2} below.

The eigenvalues of the matrix polynomial \eqref{eq22} can be determined from one of its companion matrices. For example, note that 
\begin{equation}   \label{comp_form}
 \det Q(z)=0 \qquad \iff \qquad \det (C-zD)=0
\end{equation}
for
\[
C= \begin{bmatrix}
0 &  I\\
-\mathbf{A}_2 & 2\mathbf{A}_1
 \end{bmatrix} \qquad \text{and} \qquad D= 
 \begin{bmatrix}
I &  0\\
0 & \mathbf{A}_0
 \end{bmatrix}.
\]
Indeed the assertion that $Q(z)$ is singular, is equivalent to the  existence of $\underline{u}\neq 0$ such that 
\[
\mathbf{A}_2\underline{u}-2z\mathbf{A}_1\underline{u}+z^2\mathbf{A}_0\underline{u}=0.
\]
Denoting $\underline{v}=z\underline{u}$, this can be re-written as \[\mathbf{A}_2\underline{u}-2\mathbf{A}_1\underline{v}+z\mathbf{A}_0\underline{v}=0.\]
In turns, the latter is equivalent to
\begin{equation}    \label{eq27}
  \begin{bmatrix}
0 &  I\\
-\mathbf{A}_2 & 2\mathbf{A}_1
 \end{bmatrix}
\begin{bmatrix}
\underline{u} \\
\underline{v}
 \end{bmatrix} 
 =z
\begin{bmatrix}
I &  0\\
0 & \mathbf{A}_0
 \end{bmatrix} 
\begin{bmatrix}
\underline{u} \\
\underline{v}
 \end{bmatrix},
\end{equation}
as needed for the verification of \eqref{comp_form}.

The different conjugate pairs belonging to $\x{Spec}_2(A,\mathcal{L})$ provide 
information about different components of $\x{Spec}(A)$. In order to see how these two sets are related, we consider the following \emph{upper approximation of the distance from any point}  $z\in \mathbb{C}$ \emph{to} $\spec(A)$, see \cite{1998Davies}. 
Let $F:\mathbb{C}\longrightarrow [0,\infty)$ be given by 
\begin{equation} \label{eq23}
 F(z)=\min_{0\neq v \in \mathcal{L}} \frac{\|(z-A)v \|}{\|v\|}.
\end{equation}
Then $F(z)$ is an upper bound for the \emph{Hausdorff distance} from $z$ to the spectrum of $A$,
\[
    \x{dist}[z,\spec(A)]=\min \{|z-\lambda|:\lambda\in \spec(A) \}.
\]
For completeness we include a detailed proof of this well-known assertion.

\begin{lem} \label{lem21}
 For any $z\in \mathbb{C}$, 
\begin{equation} \label{eq24}
 F(z) \geqslant \x{dist}\left[z,\spec(A)\right]. 
\end{equation}
\end{lem}
\begin{proof}
If $z\in \spec(A)$ there is nothing to prove as $F(z)\geq 0$. Assume that $z\not \in \spec(A)$.
By the fact that $\mathcal{L}\subset \dom(A)$, it follows directly that 
\begin{align*}
 F(z)&\geq\inf_{0\not=v \in \dom (A)} \frac{\|(z-A)v\|}{\|v\|}=
\left(\sup_{0\not=u \in \mathcal{H}} \frac{\|(z-A)^{-1}u\|}{\|u\|}\right)^{-1} \\
&=\|(z-A)^{-1}\|^{-1}=\x{dist}[z,\spec(A)].
\end{align*}
The latter inequality is a consequence of the fact that $A=A^*$, see e.g. \cite[Lemma~1.3.2]{1995Davies}.
\end{proof}

According to this lemma, $F(x)$ can only be small whenever $x\in \RR$ is close to $\spec(A)$.  We can establish a precise connection between the second order spectrum of $A$ relative to $\mathcal{L}$ and $F(x)$ by means of the following function,
\begin{equation*}
 G(z)=\min_{0\not=\underline{v}\in \mathbb{C}^n} \frac{\|Q(z)\underline{v}\|}{\|\underline{v}\|} \qquad \text{for }z\in \mathbb{C}.
\end{equation*}
Clearly,
\[
      \spec_2(A,\mathcal{L})= \{z\in \mathbb{C}:G(z)=0 \}.
\]

\begin{lem} \label{lel22new}
Assume that $\{b_j\}_{j=1}^n$ is an orthonormal set, so that $\mathbf{A}_0=I$. Then
\[
        G(x)=F(x)^2 \qquad \forall x\in \RR.
\]
\end{lem}
\begin{proof}
Firstly note that
\begin{align*}
\|(x-A)v\|^2&=\langle (x-A)v, (x-A)v\rangle\\
&=\left\langle \sum_{j=1}^n \hat{v}(j) (x-A)b_j,\sum_{k=1}^n \hat{v}(k) (x-A)b_k\right\rangle \\
&=\sum_{jk=1}^n \hat{v}(j)\overline{\hat{v}(k)} \frak{a}^2_x (b_j,b_k) =\langle Q(x) \underline{v},\underline{v}\rangle.
\end{align*}
Then $Q(x)$ is non-negative for $x \in \RR$ and
\begin{align*}
F(x)^2&=\min_{\underline{v}\in \mathbb{C}^n} \frac{\langle Q(x) \underline{v},\underline{v}\rangle}{\|\underline{v}\|^2} \\
&=\min\left[ \spec Q(x)\right]=\left(\min\left[ \spec Q(x)^2\right]\right)^{1/2} \\
&= \left(\min_{\underline{v}\in \mathbb{C}^n} \frac{\langle Q(x)^2 \underline{v},\underline{v}\rangle}{\|\underline{v}\|^2}  \right)^{1/2}  \\ &= \left(\min_{\underline{v}\in \mathbb{C}^n} \frac{\|Q(x) \underline{v}\|^2}{\|\underline{v}\|^2}  \right)^{1/2}=G(x).
\qedhere
\end{align*}
\end{proof}

By virtue of this lemma and the fact that both $F(z)$ and $G(z)$ are continuous, if  $G(z)=0$ for $z$ close to $\RR$, then one should expect that $F(\Re(z))$ is small.  From \eqref{eq24}, it would then follow that $\Re(z)$ is close to the spectrum of $A$ in this case. A precise statement on this matter was first established by Shargorodsky.  Below we include a proof which is independent from the seminal work \cite{Shargorodsky}.

\begin{thm}[Shargorodsky] \label{thmshar}
For $\mu\in \spec_2(A,\mathcal{L})$, let 
\begin{align*}
\mu_{\mathrm{up}}&=\Re(\mu)+|\Im(\mu)| \quad \text{and} \quad
\mu_{\mathrm{low}}=\Re(\mu)-|\Im(\mu)|. 
\end{align*}
Then 
\[
[\mu_{\mathrm{low}},\mu_{\mathrm{up}}]\cap \spec(A) \neq \varnothing.
\]
\end{thm}
\begin{proof}
Let $\mu=\alpha+i\beta$ for $\alpha,\beta \in \RR$. Then
\begin{equation} \label{eq25}
 \langle (\mu-A)u,(\bar{\mu}-A)v\rangle=\frak{a}^2_\alpha(u,v)+2i\beta\frak{a}_\alpha^1(u,v)-\beta^2\frak{a}^0(u,v).
\end{equation}
Let $u \in \mathcal{L}$ be such that 
\[
   \langle (\mu-A)u,(\bar{\mu}-A)v\rangle=0 \qquad \forall v\in \mathcal{L}.
\]
Then either $\beta = 0$, in which case $\alpha \in \spec(A)$ and $u\in \ker(\alpha-A)$, or $\beta\not=0$. 
If the latter happens, we use \eqref{eq25} for $u=v$ to get 
\[
 \|(\alpha-A)u\|^2-\beta^2\|u\|^2+2i\beta \frak{a}^1_\alpha(u,u)=0,
\]
so $\beta^2=\frac{\|(\alpha-A)u\|^2}{\|u\|^2}$ and $ \frak{a}^1_\alpha(u,u)=0$.
Hence, according to \eqref{eq24}, 
\begin{align*}
 |\beta|&= \frac{\|(\alpha-A)u\|}{\|u\|} \geq F(\alpha) \\ &\geq \x{dist}[\alpha,\spec(A)].
\end{align*}
\end{proof}

A numerical method for computing bounds on points in $\x{Spec}(A)$ arises naturally from this theorem.
Given $\mathcal{L}\subset \dom(A)$, find those conjugate pairs of $\spec_2(A,\mathcal{L})$ which are close to
the real line. These will give small intervals containing points in $\spec(A)$. 
This method has been referred-to as the \emph{quadratic method} and its implementation in concrete models has been
examined in \cite{Levitin,2006Boulton, 2008Boulton,2010Boulton, 2010strauss}. In sections~\ref{convergencesec}  and 
\ref{schrosec}
we give conditions on $\mathcal{L}$ ensuring spectral exactness, that is convergence of part of $\spec_2(A,\mathcal{L})$ to $\x{Spec}(A)$ in some precise regime $n\to \infty$.

The quadratic method is based on the idea that the truncated operator $(A-z)^2\!\upharpoonright_\mathcal{L}$ will be non-invertible for $z$ close to the real line, only if $z$ is close to the spectrum of $A$. In turns, there is an underlying mapping theorem for quadratic projected operators, as described by \cite[Lemma~2.6]{2010Boulton}, which is not available in general for the classical Galerkin method \cite[Remark~4]{2012Boulton}. This mapping theorem is described next and it will be crucial in our examination of convergence in Section~\ref{convergencesec}.  

\begin{lem}   \label{mappinglemma}
Let $a\in \mathbb{R}\setminus \spec(A)$. Then
\[
      z,\overline{z}\in \spec_2(A,\mathcal{L}) \qquad \iff \qquad w,\overline{w}\in \spec_2(B,\mathcal{G})
\] 
where $w=(z-a)^{-1}$, $B=(A-a)^{-1}$ and $\mathcal{G}=(A-a)\mathcal{L}$.
\end{lem}
\begin{proof}
Without loss of generality we can assume that $z$ (and hence $w$) are non-real. Otherwise, the stated result follows directly from the Spectral Mapping Theorem combined with Theorem~\ref{thmshar}.

Let $u,\,v \in \mathcal{L}$. Let $\tilde{u}=(A-a)u\in \mathcal{G}$ and $\tilde{v}=(A-a)v \in \mathcal{G}$. Note that
\begin{align*}
\langle (B-w)\tilde{u},(B-\overline{w})\tilde{v} \rangle  =   (z-a)^{-2} \langle(zI-A)u,(\overline{z}I-A)v\rangle.
\end{align*}
Since $z\not= a$, the left hand side will vanish if and only if the second term on the right vanish.
From the definition of the second order spectrum, this implies directly the desired statement.
\end{proof}

At first sight it might seem that the quadratic method is numerically too expensive for practical purposes, as it reduces to computing conjugate pairs which are the eigenvalues of a quadratic matrix polynomial problem. It is indeed true that, ultimately, the problem reduces to computing the eigenvalues of a companion matrix such as \eqref{eq27}, and that this matrix is not normal, so numerical calculation of its eigenvalues is intrinsically more unstable than computing the eigenvalues of a hermitean matrix problem.  On the other hand however, as suggested by Theorem~\ref{thmshar}, the method is extremely robust. Given any linear subspace $\mathcal{L}$ of the domain of $A$, projection onto the real line of $\spec_2(A,\mathcal{L})$
always provides true information about the spectrum of $A$.   


\section{Spectral exactness} \label{convergencesec}

Assume that a sequence of subspaces $\mathcal{L}$ increases towards $\dom(A)$. We now establish precise condition on this sequence, in order to ensure that points in the second order spectrum of $A$ relative to $\mathcal{L}$ approach the real line and hence the spectrum.

Spectral exactness of the quadratic method has been examined in detail in \cite{200666Boulton,20077Boulton,2010Boulton}. The result \cite[Theorem~3.4]{2010Boulton} provides a precise estimate on the convergence rate of the second order spectrum to the discrete spectrum. As it turns, see \cite[\S4(b)]{2010Boulton}, the rate derived from this result is generally sub-optimal.  Our main goal now is to improve the estimate on the order of this convergence. The two crucial ingredients in our proof below are the original statement of convergence \cite[Theorem~3.4]{2010Boulton} for the case of a bounded operator and the mapping property determined by Lemma~\ref{mappinglemma}. 

 Without further mention below the open ball of radius $\rho>0$ in the complex plane centred at $b\in \RR$ will be $\mathbb{B}(\rho,b)$.  We omit the proof of the following crucial statement, as it is a direct consequence of \cite[Theorem~3.4]{2010Boulton}. 

\begin{thm} \label{convergencequadraticboundedthm}
Let $B$ be a bounded operator. Let $\mu\in \spec(B)$ be an isolated eigenvalue and 
$\mathcal{E}=\{ \phi_1,\ldots,\phi_m\}\subset \ker(B-\mu)$ be an orthonormal set. Let $\mu_{\pm}\in \mathbb{R}$ be such that
\[
     \mu_-<\mu<\mu_+ \quad \text{and} \quad [\mu_-,\mu_+]\cap \spec(B)=\{\mu\}.
\]
There exist $\kappa>0$ and $\delta_0>0$
only dependant on $\mu_\pm$, $\mathcal{E}$ and $B$, ensuring the following.  If the trial subspace $\mathcal{G}\subset \mathcal{H}$ is such that
\[
     \max_{\phi\in \mathcal{E}} \min_{v\in \mathcal{G}}\|v-\phi\| \leq \delta
\]
for some $0<\delta<\delta_0$, then
\[
        \spec_2(B,\mathcal{G})\cap \mathbb{B}\left(\frac{\mu_+-\mu_-}{2},\frac{\mu_++\mu_-}{2}  \right)\subset \mathbb{B}(\kappa \delta^{1/2},\mu).
\]
\end{thm}

In Theorem~\ref{convergencequadraticboundedthm} as well as in the next corollary,  the corresponding isolated eigenvalue might be of infinite multiplicity and the orthonormal set $\mathcal{E}$ might or might not be a basis of the eigenspace. The following is the main result of this paper. 

\begin{col} \label{convquadgeneral}
Let $\lambda\in \spec(A)$ be an isolated eigenvalue and let 
$\mathcal{E}=\{ \phi_1,\ldots,\phi_m\}\subset\ker(A-\lambda)$ be an orthonormal set. 
Let $\lambda_{\pm}\in \mathbb{R}$ be such that
\[
     \lambda_-<\lambda<\lambda_+ \quad \text{and} \quad [\lambda_-,\lambda_+]\cap \spec(A)=\{\lambda\}.
\]
 There exist $K>0$ and $\varepsilon_0>0$ only dependant on $\lambda_{\pm}$, $\mathcal{E}$ and $A$, ensuring the following.  If $\mathcal{L}\subset \dom(A)$ is such that
\[
      \max_{\phi\in \mathcal{E}} \min_{u\in \mathcal{L}} \left( \|u-\phi\|+\|A(u-\phi)\| \right)\leq \varepsilon
\]
for some $0<\varepsilon<\varepsilon_0$, then
\[
        \spec_2(A,\mathcal{L})\cap \mathbb{B}\left(\frac{\lambda_+-\lambda_-}{2},\frac{\lambda_++\lambda_-}{2}  \right)\subset \mathbb{B}(K \varepsilon^{1/2},\lambda).
\]
\end{col}
\begin{proof}
Let 
\[
     a\in \RR\setminus \Big(\spec(A) \cup[\lambda_{-},\lambda_+]\Big).
\]
The existence of $a$ is ensured by the fact that $\lambda_{\pm}\not\in \spec(A)$ together with the fact that $\spec(A)$ is closed.
Let $B=(A-a)^{-1}$. We combine Lemma~\ref{mappinglemma} with Theorem~\ref{convergencequadraticboundedthm} for  $\mathcal{G}=(A-a)\mathcal{L}$ and $\mu=(\lambda-a)^{-1}$, as follows. Observe that $\mu$ is an isolated eigenvalue of $B$ and that $\mathcal{E}$ are associated eigenfunctions.

Let
\[
   f(z)=\frac{1}{z-a} \qquad \text{and} \qquad g(w)=\frac{1+wa}{w}.
\]
Then $f(z)$ is a M\"obius transformation and $g(w)$ is its inverse. Moreover
\[
  f(\spec(A)\cup \{\infty\})=\spec(B) \quad \text{and} \quad g(\spec(B))=\spec(A)\cup \{\infty\}.
\]
According to Lemma~\ref{mappinglemma},  
\[
  f(\spec_2(A,\mathcal{L}))=\spec_2(B,\mathcal{G}) \quad \text{and} \quad g(\spec_2(B,\mathcal{G}))=\spec_2(A,\mathcal{L}).
\]
Note that $f(z)$ maps the disk
\[
\mathbb{B}   \left(\frac{\lambda_+-\lambda_-}{2},\frac{\lambda_++\lambda_-}{2}  \right) 
\] 
into a disk with diameter a segment containing $\mu$ in its interior. Let $(\mu_-,\mu_+)\subset \RR$ denote such a segment. Then
\[
        \mathbb{B}\left(\frac{\mu_+-\mu_-}{2},\frac{\mu_++\mu_-}{2}  \right)= f\left( \mathbb{B}   \left(\frac{\lambda_+-\lambda_-}{2},\frac{\lambda_++\lambda_-}{2}  \right)  \right)
\]
and
\[
     [\mu_-,\mu_+]\cap \spec(B)=\{\mu\}.
\]

Let $\delta_0$ and $\kappa$ be the constants found by Theorem~\ref{convergencequadraticboundedthm} with the above data. 
Let $d=|\lambda-a|>0$ and
\[
    \varepsilon_0=\frac{d}{1+|a|} \min\left\{\delta_0,\frac{|\mu|^2}{2\kappa^2}\right\}>0.
\]
If 
\[
      \|u-\phi\|+\|A(u-\phi)\|<\varepsilon <\varepsilon_0
\]
for $u\in \mathcal{L}$ and $\phi\in \mathcal{E}$, then fixing $\tilde{u}=\frac{1}{\lambda-a} u$ gives
\begin{align*}
     \|(A-a)\tilde{u} -\phi\|&=  \|(A-a)\tilde{u} -\frac{\lambda-a}{\lambda-a}\phi\| \\
        &=\frac{1}{|\lambda-a|} \|(A-a)(u-\phi)\|    \\
        &\leq \frac{1}{d} \left(|a|\|u-\phi\|+\|A(u-\phi)\| \right) \\
        & \leq\frac{1+|a|}{d} \varepsilon <\delta_0.
\end{align*} 
Define $\delta=\frac{1+|a|}{d} \varepsilon$. Then $v=(A-a) \tilde u\in \mathcal{G}$ yields 
\[\|v-\phi\|\leq \delta <\delta_0.\] Thus the hypothesis of this corollary implies the hypothesis of 
Theorem~\ref{convergencequadraticboundedthm}.

Now the conclusion of Theorem~\ref{convergencequadraticboundedthm} gives
\begin{align*}
     \spec_2(A,\mathcal{L}) & \cap \mathbb{B}   \left(\frac{\lambda_+-\lambda_-}{2},\frac{\lambda_++\lambda_-}{2}  \right)  \\& = g\left(  \spec_2(B,\mathcal{G}) \cap  \mathbb{B}\left(\frac{\mu_+-\mu_-}{2},\frac{\mu_++\mu_-}{2}  \right) \right) \\
     &\subset g\left(\mathbb{B}(\kappa \delta^{1/2},\mu)\right) .
\end{align*}
The diameter of the latter is
\[
 \frac{2\kappa K_1^{1/2} }{\mu^2-\kappa^2 K_1\varepsilon } \varepsilon^{1/2}\qquad
\text{for} \qquad K_1=\frac{1+|a|}{d}.
\]
In turns, the definition of $\varepsilon_0$ ensures that
\[
      \frac{2\kappa K_1^{1/2} }{\mu^2-\kappa^2 K_1\varepsilon }\leq 
       K \qquad \text{for} \qquad K=\frac{4\kappa K_1^{1/2} }{\mu^2}.
\]
As $\lambda\in g\left(\mathbb{B}(\kappa \delta^{1/2},\mu)\right)$, then
\[
 g\left(\mathbb{B}(\kappa \delta^{1/2},\mu)\right) \subset\mathbb{B}(K\varepsilon^{1/2},\lambda).
\]
This ensures the conclusion claimed in the corollary.
\end{proof}

Observe that $K\to \infty$ and $\varepsilon_0\to 0$ in the regime $|\lambda-a|\to 0$.
We will see in Section~\ref{numerical} that the conclusion of this corollary is sub-optimal in the power of the parameter $\varepsilon$. However, as mentioned earlier,  it supersedes significantly \cite[Theorem~3.4]{2010Boulton} in the case of $A$ unbounded.


\section{Eigenvalue bounds for Schr{\"o}dinger operators} \label{schrosec}

We now examine the implementation of the quadratic method in a particularly simple instance. Set $A=H$ a  one-dimensional Schr{\"o}dinger operator. We consider that the trial subspaces $\mathcal{L}$ are constructed  via the finite element method on a large, but finite, segment. Under standard assumptions on the convergence of the finite element subspaces as the mesh refines and the length of the segment grows, we determine an upper bound on the precise convergence rate at which conjugate pairs in the second order spectra converge to the eigenvalues of $H$.

We begin by fixing the precise setting for the operator $H$. Let
\begin{equation*} \label{eq61}
 H u(x)=-u''(x)+V(x) u(x)  \qquad x\in (-\infty,\infty)
\end{equation*}
acting on $L^2(\mathbb{R})$. We assume that the potential $V(x)$ is real-valued, continuous and $V(x) \to \infty$  as $|x| \to \infty$. These conditions ensure that the operator $H$ is self-adjoint on a domain defined via Friedrich's extensions and it has a compact resolvent \cite[Theorem XIII.67]{1980Barry}.
The domain of closure of the quadratic form associated to $H$ is
\[
    \dom (\frak{a}^{1}) = W^{1,2}(\mathbb{R}) \cap \{u\in L^2(\mathbb{R}):\|V^{1/2}u\|<\infty\}.
\]
Note that this is the intersection of the maximal domains of the momentum operator and the operator of multiplication by $|V|^{1/2}$.

The conditions on the potential imply that $V(x) \geqslant b_0>-\infty$ for all $x \in \mathbb{R}$ and a suitable constant  $b_0 \in \mathbb{R}$. Then $H$ is bounded below in the sense of quadratic forms, $H \geq b_0$. Without loss of generality we assume below that $b_0>0$.

By compactness of the resolvent, $H$ has a purely discrete spectrum, comprising only eigenvalues accumulating at $+\infty$ and a basis of eigenfunctions. Moreover,  by the fact that we are in one space dimension, we know that all these eigenvalues are simple. The eigenfunctions\footnote{Recall that the potential is continuous.} are $C^\infty$ and they decay exponentially fast at infinity \cite[Theorem C.3.3]{1982Barry}. 
We denote
\[
    \x{Spec}(H)=\{\lambda_1< \lambda_2<\ldots \}
\]
and let the orthonormal basis $\{\psi_j\}_{j=1}^\infty$ of $L^2(\mathbb{R})$ be such that \[H\psi_j=\lambda_j\psi_j.\] Without further mention,  below we often suppress the index $j$ from the eigenvalue and the eigenfunction, when the context allows it.

Let us describe the construction of the trial spaces.  Let $L > 0$. Consider the restricted operator
\begin{equation*} \label{eq62}
  H_L u(x)=-u''(x)+V(x) u(x)    \qquad x\in (-L,L),
\end{equation*}
subject to Dirichlet boundary conditions:   $u(-L)=u(L)=0$. As $L\to \infty$, we expect that the spectrum of $H_L$ approaches the spectrum of $H$. In fact, according to Theorem~\ref{lem63} below, this turns out to happen exponentially fast (in $L$) for individual eigenvalues. 

Similarly to $H$, the operator $H_L$ acts on a domain also defined via Friedrich's extensions. Denote by $\frak{a}^{1,L}$ the quadratic form associated to $H_L$. The domain of closure of  $\frak{a}^{1,L}$ is
\[
 \dom(\frak{a}^{1,L})=W^{1,2}_0(-L,L)
\]
see \cite[Theorem~VI.2.23 and VI.4.2]{kato}. 

As $b_0>0$, the forms $\frak{a}^{1}$ and $\frak{a}^{1,L}$ are positive definite. Hence the quantities
\[
     \frak{a}^{1}(u,u)^{1/2} \qquad \text{and} \qquad   \frak{a}^{1,L}(u,u)^{1/2}
\]
define norms in $\dom(\frak{a}^{1})$ and $\dom(\frak{a}^{1,L})$ respectively.

Without further mention, everywhere below we assume that (additionally to the conditions above), $V(x)$ is such that for every $b>0$ there exists a constant $k_b>0$ ensuring
\[
    |V(x)| \leqslant k_b \mathrm{e}^{b|x|}\qquad \forall x \in \mathbb{R}.
\]
The following statement is well known. We include its proof, in order to keep a complete exposition of the subject. 

\begin{lem}\label{lem61}
There exist constants $c > 0$ and $a>0$ only dependant on $j\in\mathbb{N}$ and the potential $V$, such that 
\begin{equation*}
|\psi_j(x)|\leqslant c \mathrm{e}^{-a|x|} \qquad \text{and} \qquad  |\psi_j{''}(x)|\leqslant c \mathrm{e}^{-a|x|} \qquad \forall x \in \mathbb{R}.
\end{equation*}
 \end{lem}
\begin{proof}
The identity
\[ -\psi{''}+V(x)\psi -\lambda \psi=0 \]
implies that
\[ 
 |\psi{''}(x)|=|V(x)-\lambda||\psi(x)|.
\]
By \cite[Theorem C.3.3]{1982Barry}, we know 
\[|\psi(x)|\leqslant \widetilde{c}~ \mathrm{e}^{-\widetilde{a}|x|}.\]
Let $a=\frac{\widetilde{a}}{2}$  and $b=\frac{a}{2}$. Then
\begin{align*}
 |\psi{''}(x)| &\leqslant \widetilde{c} \; |V(x)-\lambda| \; \mathrm{e}^{-\widetilde{a}|x|}\\
& \leq \widetilde{c} \left(|V(x)|+ |\lambda|\right) \mathrm{e}^{-\widetilde{a}|x|}\\
& \leq \widetilde{c}\left(k_b e^{(b-a)|x|} + |\lambda| \mathrm{e}^{-a|x|} \right) \mathrm{e}^{-a|x|}.
\end{align*}
\end{proof}

Here and everywhere below $a>0$ is a constant, found according to Lemma~\ref{lem61}, which might depend on $j$.

\begin{lem} \label{lem62}
For any $L$ sufficiently large,  there exists $c>0$ only dependant on $j$ such that
\[
    \int_{\mathbb{R}\setminus[-L,L]} |\psi'_j(x)|^2 \dx  \leq c \, \mathrm{e}^{-2aL}.
\]
\end{lem}
\begin{proof}
\begin{align*}
 &\int_{\mathbb{R}\setminus[-L,L]} |\psi'(x)|^2 \dx  \\
 &=\left|\int_{-\infty}^{-L}\!\!\!+\int_{L}^{\infty} \left(-\psi''(x)\overline{\psi(x)}\right)\dx +\left[\psi'(x)\overline{\psi(x)}\right]^{-L}_{-\infty}+\left[\psi'(x)\overline{\psi(x)}\right]^{\infty}_{L} \right|\\
&\leq \int_{-\infty}^{-L}\!\!\!+\int_{L}^{\infty} \left(|\psi''(x)| |\psi(x)|\right)\dx \!+\!\left[|\psi'(x)||\psi(x)|\right]^{-L}_{-\infty}\!+\!\left[|\psi'(x)||\psi(x)|\right]^{\infty}_{L}\\
&\leq  c_1\, \mathrm{e}^{-3aL}+ c_2\, \mathrm{e}^{-3aL}+ c_3\, \mathrm{e}^{-2aL}+c_4\, \mathrm{e}^{-2aL}
\end{align*}
\end{proof}

In the following statements the cutoff function
\[
  h_L(x)= \begin{cases} 0 &\mbox{if }~ x\in(-\infty,-2L] \\
\exp\left({1-\frac{1}{1-(\frac xL +1)^2}}\right) & \mbox{if }~ x\in[-2L,-L] \\
1 & \mbox{if }~ x\in[-L,L]\\
\exp\left({1-\frac{1}{1-(\frac xL -1)^2}}\right) & \mbox{if }~ x\in[L,2L] \\
0 &\mbox{if }~ x\in[2L,\infty) 
\end{cases}
\]
and its derivative is
 \begin{equation*}
  h_L'(x)= \begin{cases} 0 &\mbox{if }~ x\in(-\infty,-2L] \\
-\frac{2(L+x)L^2}{x^2(2L+x)^2} h(x)& \mbox{if }~ x\in[-2L,-L] \\
0 & \mbox{if }~ x\in[-L,L]\\
\frac{2(L-x)L^2}{x^2(2L-x)^2} h(x) & \mbox{if }~ x\in[L,2L] \\
0 &\mbox{if }~ x\in[2L,\infty).
\end{cases}
\end{equation*}
Note that, for any function $v\in \dom (\frak{a}^{1})$, 
\[
h_Lv \in \dom (\frak{a}^{1,2\tilde{L}}) \qquad\qquad \forall \tilde{L}\geq L
\]
and also $h_Lv \in \dom (\frak{a}^{1})$.

\begin{lem} \label{lem63_max}
Fix $j\in \mathbb{N}$. Let 
\[
\psi_k^L=h_L\psi_k \qquad \text{and} \qquad U_j(L)=\x{Span}\{\psi^L_k\}_{k=1}^j.
\] There exist sufficiently large constants $L_j>0$ and $c_j>0$, ensuring the following. 
\begin{enumerate}
\item \label{prop1}  $\dim U_j(L)=j$ for all $L>L_j$.  
\item \label{prop2}  For any $v\in U_j(L)$ of unit $L^2$-norm, there exists $\phi\in \x{Span}\{\psi_k\}_{k=1}^j$
such that $\|\phi\|=1$ and
\[
       \frak{a}^{1}(v-\phi,v-\phi) \leq c_j e^{-2aL} \qquad \qquad \forall  L>L_j.
\]
\end{enumerate}
\end{lem}
\begin{proof} Since the $\psi_k$ are linearly independent in $L^2(\mathbb{R})$ and they are exponentially small for large $x$, then the set $\{\psi_k^L\}_{k=1}^j$ is linearly independent for $L>\tilde{L}$ where the latter is large enough. 

The existence of $c_j$ is ensured as follows. Let $v=\sum_{k=1}^j \alpha_k \psi_k^L$. Define $\tilde{\phi}=\sum_{k=1}^j \alpha_k \psi_k$ and $\phi=\frac{1}{\|\tilde{\phi}\|}\tilde{\phi}$. 
By Lemma~\ref{lem61}, we have
\[
    \|\psi_k^L-\psi_k \|^2 \leq c_{1}(k) e^{-2aL}.
\]
Then, 
\begin{align*}
   \|v-\tilde{\phi}\| &= \left\|\sum_{k=1}^j\alpha_k(\psi_k^L-\psi_k)\right\|
\leq \sum_{k=1}^j |\alpha_k| \|\psi_k^L-\psi_k\| \\
&\leq \sum_{k=1}^j |\alpha_k| c_{1}(k)^{1/2} e^{-aL} \leq \tilde{c}_j e^{-aL}
\end{align*}
so that $1- \tilde{c}_j e^{-aL}\leq \|\tilde{\phi}\|\leq 1+ \tilde{c}_j e^{-aL}$.
Now, from Lemma~\ref{lem62}, we get
\[
    \frak{a}^1(\psi_k^L-\psi_k,\psi_k^L-\psi_k)\leq c_2(k) e^{-2aL}.
\]
Then, from the fact that $\sqrt{\frak{a}^1}$ is a norm in its domain, an application of the triangle inequality yields
\[
   \frak{a}^1(v-\tilde{\phi},v-\tilde{\phi}) ^{1/2}  \leq \tilde{\tilde{c}}_j e^{-aL}.
\]
Here $\tilde{c}_j$ and $\tilde{\tilde{c}}_j$ are independent of the $\alpha_k$. Thus
\[
 \frak{a}^1(v-\phi,v-\phi) ^{1/2} \leq c_j e^{-aL}
\]
as needed.
\end{proof}

We now show that the eigenvalues for a finite $L$ are exponentially close to those for the infinite $L$.

\begin{thm}  \label{lem63}
 Let $\lambda_j$ be the $j$th eigenvalues of $H$ and $\lambda_j^L$ be the $j$th eigenvalue of $H_L$. For $L$ sufficiently large, there exists a constant $c_j>0$ independent of $L$ such that 
\[
                     \lambda_j < \lambda_j^{L} < \lambda_j+c_j\mathrm{e}^{-2a L}.
\]
\end{thm}

\begin{proof}   Let $\widehat{D}_L \subset \dom(\frak{a}^{1,\infty})$ be
\[
 \widehat{D}_L=\left\{u:\mathbb{R}\rightarrow \mathbb{C}~:~ u\!\upharpoonright_{[-L,L]}\in \dom(\frak{a}^{1,L})  \text{ and } u\!\upharpoonright_{(-\infty,-L]\cup[L,\infty)}=0\right\}.
\]
Then
\begin{align*}
 \lambda_j &=\min_{\substack {V\subset \dom(\frak{a}^{1,\infty})\\ \x{dim} V=j}}~ \max_{\substack {u\in V \\ u \neq 0}} ~ \frac{\frak{a}^{1,\infty}(u,u)}{\langle u,u\rangle}
\leq \min_{\substack {V\subset \widehat{D}_{2L}\\ \x{dim} V=j}}~ \max_{\substack {\tilde{u}\in V \\ \tilde{u} \neq 0}} ~ \frac{\frak{a}^{1,\infty}(\tilde{u},\tilde{u})}{\langle \tilde{u},\tilde{u}\rangle}\\
&= \min_{\substack {V\subset  \dom(\frak{a}^{1,2L})\\ \x{dim} V=j}}~ \max_{\substack {\tilde{u}\in V \\ \tilde{u} \neq 0}} ~ \frac{\frak{a}^{1,2L}(\tilde{u},\tilde{u})}{\langle \tilde{u},\tilde{u}\rangle}=\lambda_j^{2L}.
\end{align*}

Let $v_j\in \dom(\frak{a}^{1,2L})$ be the extremal vector of norm 1 such that
\[
       \frak{a}^{1}(v_j,v_j)=\frak{a}^{1,2L}(v_j,v_j)=\max_{\substack {v\in U_j(L) \\ v \neq 0}} ~ \frac{\frak{a}^{1,2L}(v,v)}{\langle v,v\rangle}
\]
The property~\ref{prop1} from Lemma~\ref{lem63_max}, implies that 
\[
 \lambda_j^{2L} \leq \frak{a}^{1}({v}_j,{v}_j).
\]
Then, according to the property~\ref{prop2} from the same lemma, there exists 
\[
    \phi \in \x{Span}\{\psi_l\}_{l=1}^j \qquad \|\phi\|=1
\]
 such that
\begin{align*}
 \lambda_j^{2L} \leq \frak{a}^{1}(\phi,\phi) +c_j\mathrm{e}^{-2aL}\leq \frak{a}^{1}(\psi_j,\psi_j) + c_j\mathrm{e}^{-2aL}=\lambda_j+c_j\mathrm{e}^{-2aL}.
\end{align*}
\end{proof}

Let $\Xi$ be an equidistant partition of $[-L,L]$ into $n$ sub-intervals $\mathcal{I}_l=[x_{l-1},x_l]$ of length $h=\frac{2L}{n}$. Let $\mathcal{L}=\mathcal{L}_L^h=V_h(k,r,\Xi)$ where 
\begin{equation}    \label{fespaces}
 V_h(k,r,\Xi)=\left\{v\in C^k(-L,L): \begin{aligned}& v \!\!\upharpoonright_{ \mathcal{I}_l} \in P_r(\mathcal{I}_l)  \qquad 1 \leq l \leq n \\ & v(-L)=v(L)=0 
 \end{aligned} \right\}
\end{equation}
is the finite element space generated by C$^k$-conforming elements of order $r$ subject to Dirichlet boundary conditions. Here we require $k\geq 1$ and $r \geq 3$, to ensure that $\mathcal{L} \subset \dom(\frak{a}^{2,L})$.

\begin{thm} \label{thm62}
Fix $j\in \mathbb{N}$.
There exist $L_0 > 0$ large enough and $h_0>0$ small enough, such that the following is satisfied. For $L>L_0$ and $h<h_0$, we can always find $u_j \in \mathcal{L}_L^h$ such that 
\begin{enumerate}
 \item \label{c1} $\langle u_j-\psi_j, u_j-\psi_j \rangle \leq \epsilon^0(h,L)$
  \item \label{c2} $\langle H(u_j-\psi_j), u_j-\psi_j \rangle \leq \epsilon^1(h,L)$
  \item \label{c3} $\langle H(u_j-\psi_j), H(u_j-\psi_j) \rangle \leq \epsilon^2(h,L)$
\end{enumerate}
where 
\begin{align*}
\epsilon^0(h,L)&=c_{10} \mathrm{e}^{-4aL}+c_{20} h^{2(r+1)} \\
\epsilon^1(h,L)&=c_{11} \mathrm{e}^{-2aL}+c_{21} h^{2r} \\
\epsilon^2(h,L)&=c_{12}\mathrm{e}^{-2aL}+c_{22}h^{2(r-1)}.
\end{align*}
The constants $c_{nk}>0$ are dependant on $j$, but are independent of $L$ or $h$.
\end{thm}
\begin{proof}
Below we repeatedly use the estimate
\[
    \|v-v_h\|_{H^p(-L,L)} \leq \tilde{c} h^{r+1-p}
\]
where $v_h\in V_h(k,r,\Xi)$ is the interpolate of $v\in C^k\cap H_{\Xi}^{r+1}(-L,L)$.
See \cite[Theorem~3.1.6]{1978Ciarlet}. We set $u=\psi_h$.

For the property~\ref{c1}, observe that
 \begin{align*}
  \langle u-\psi, u-\psi \rangle&=\int_{-\infty}^{\infty} |u-\psi|^2\dx \\
&=\int_{-\infty}^{-L}\!\!\!+\int_{L}^{\infty} |u-\psi|^2\dx +\int_{-L}^{L}|u-\psi|^2\dx \\
&=\int_{-\infty}^{-L}\!\!\!+\int_{L}^{\infty} |\psi|^2\dx +\int_{-L}^{L}|u-\psi|^2\dx \\
&\leq c_{10} e^{-4aL}+c_{20} h^{2(r+1)}.
 \end{align*}

For the property~\ref{c2}, observe that
\begin{align*}
&\langle H(u-\psi), u-\psi \rangle\leq \int_{-\infty}^{\infty} \left(|u'-\psi'|^2+|V(x)||u-\psi|^2\right)\dx \\
&=\int_{-\infty}^{-L}\!\!\!+\int_{L}^{\infty}|\psi'|^2\dx +\int_{-\infty}^{-L}\!\!\!+\int_{L}^{\infty}|V(x)||\psi|^2\dx 
+ \int_{-L}^{L}|u'-\psi'|^2\dx  \\
& \qquad +\int_{-L}^{L}|V(x)||u-\psi|^2\dx \\
&\leq c_{31} \mathrm{e}^{-2aL}+c_{32} \mathrm{e}^{- \frac72aL}+c_{33} h^{2r} +c_{34} h^{2(r+1)}.
\end{align*}

For the property~\ref{c3}, observe that
\begin{align*}
&\langle H(u-\psi),H(u-\psi)\rangle\\
&\leq \int_{-\infty}^{\infty}|u''-\psi''|^2\dx +\int_{-\infty}^{\infty}|V(x)|^2|u-\psi|^2\dx \\
&\qquad +2\left[\left(\int_{-\infty}^{\infty}|V(x)|^2|u-\psi|^2\dx  \right)^{1/2} \left(\int_{-\infty}^{\infty}|u''-\psi''|^2\dx \right)^{1/2}\right]\\
&\leq\left( c_{41}\mathrm{e}^{-2aL}+c_{42}h^{2(r-1)}\right)+\left( c_{43}\mathrm{e}^{-3aL}+c_{44}h^{2(r+1)}\right)\\
&\qquad +\left[\left(c_{45}\mathrm{e}^{-3aL}+c_{46}h^{2(r+1)}\right)^{1/2} \left( c_{47}\mathrm{e}^{-2aL}+c_{48}h^{2(r-1)} \right)^{1/2}\right]\\
&= c_{41}\mathrm{e}^{-2aL}+c_{42}h^{2(r-1)}+ O(\mathrm{e}^{-3aL})+O(h^{2r}).
\end{align*}
Here we employ the Cauchy-Schwarz inequality and Newton's Generalised Binomial Theorem, as well as 
lemmas~\ref{lem61} and \ref{lem62}.
\end{proof}

By virtue of Theorem~\ref{thm62}, the hypothesis of Corollary~\ref{convquadgeneral}
is satisfied for $A=H$ and $\mathcal{L}=\mathcal{L}_L^h$, whenever $h$ is small enough and $L$ is large enough. Therefore spectral exactness is guaranteed in this setting. We summarise the crucial statement in the following corollary.

Note that the upper bound for $d$ below is the distance from $\lambda_j$ to the rest of the spectrum of $H$.

\begin{col} \label{cor62}
Fix $j\in\mathbb{N}$ and
\[
0<d<\min\{ \lambda_j-\lambda_{j-1},\lambda_{j+1}-\lambda_j  \}.
\]
There exists $L_0>0$, $h_0>0$ and $c>0$ (only dependant on $j$, $d$ and $H$) ensuring the following.  
For all $L>L_0$ and $0<h<h_0$,
\[
    \spec_2(H,\mathcal{L}_L^h)\cap \mathbb{B}\left(d,\lambda_j \right)\subset \mathbb{B}(c (h^{\frac{r-1}{2}}+e^{-\frac{aL}{2}}),\lambda_j).
\]
\end{col}
\begin{proof}
Let $u_j\in\mathcal{L}_L^h$ be as in Theorem~\ref{thm62}.
Then
 \begin{align*}
\|u_j-\psi_j\|&+\|H(u_j-\psi_j)\| \\&=\langle u_j-\psi_j, u_j-\psi_j \rangle^{1/2}
+ \langle H(u_j-\psi_j),H(u_j-\psi_j) \rangle ^{1/2}\\
&\leq \epsilon^0(h,L)^{1/2}+\epsilon^2(h,L)^{1/2}\leq \tilde{c} (h^{r-1}+e^{-\frac{aL}{2}}).
\end{align*}
Hence Corollary~\ref{convquadgeneral} ensures the claimed statement. 
\end{proof}


 \section{Benchmarks examples} \label{numerical}

Corollary~\ref{cor62} shows that the quadratic method for one-dimensional Schr{\"o}dinger operators is convergent at a rate proportional  to $h^{\frac{r-1}{2}}$ for large enough $L$, when the trial spaces are chosen to be $\mathcal{L}=\mathcal{L}_L^h$. We now explore numerically the scope of this assertion on two well-known benchmark models, the harmonic and the anharmonic oscillators. 

Let $H^{\x{har}}=H$ for $V(x)=x^2$. Then 
\[
     \x{Spec} (H^{\x{har}}) = \{2j+1: j\in \mathbb{N}\}.
\]
Let $H^{\x{anh}}=H$ for $V(x)=x^4$.  In this case a simple analytic expression for the eigenvalues is not known. All the numerical calculations shown below were performed on fixed $L=6$ and $r=3$. The trial subspaces $\mathcal{L}_6^h$ were assembled from a basis of C$^1$-conforming Hermite elements. This ensures $\mathcal{L}_6^h\subset \dom(H)$. In correspondence with \eqref{fespaces}, everywhere in this section $n$ denotes the number of subdivisions of  the segment $[-6,6]$.
 
A routine Galerkin approximation ensures upper bounds for the eigenvalues of both $H^{\x{har}}$ and $H^{\x{anh}}$. For reference, in Table~\ref{fig4} we have included the numerical estimation of these bounds for the first five eigenvalues of the corresponding operators.

\begin{table}[b]
\centerline{\begin{tabular}{ |c|r| r| }
  \hline
  $j$ &  $\lambda_j$ harmonic &  $\lambda_j$ anharmonic \\ \hline
  1 & $1.000000000000174$ & $1.060362090484841$  \\ \hline
  2 & $3.000000000001666$ &  $3.799673029810648$ \\ \hline
  3 & $5.000000000013855$ & $7.455697938053159$ \\ \hline
  4 & $7.000000000181337$ & $11.644745511679762$ \\ \hline
  5 & $9.000000002611037$ & $16.261826019859956$ \\ \hline
\end{tabular}}
  \caption{Approximation the first five eigenvalues of $H_6^{\x{har}}$ and $H_6^{\x{anh}}$ by means of the Galerkin method on cubic Hermite elements fixing $n=400$.\label{fig4}}
\end{table}

Below we examine the explicit bounds given by Theorem~\ref{thmshar} and their convergence given by Corollary~\ref{cor62}.  We compute $\spec_2(H_6,\mathcal{L}_{6}^h)$ in the same fashion as described in Section~\ref{quasec}.  All the coefficients of the matrices $\mathbf{A}_0$, $\mathbf{A}_1$ and $\mathbf{A}_2$, were found 
analytically. The precise expressions for the entries of these matrices are  found in the Appendix.

Figures~\ref{fig5_1} and \ref{fig5_2} show $\spec_2(H^{\x{har}}_6,\mathcal{L}_{6}^h)$ and  $\spec_2(H^{\x{anh}}_6,\mathcal{L}_{6}^h)$, for three different values of $n$. Clearly in both cases the second order spectra are globally approaching the spectrum. 

 \begin{figure}[H]
\centering
\includegraphics[scale=0.4]{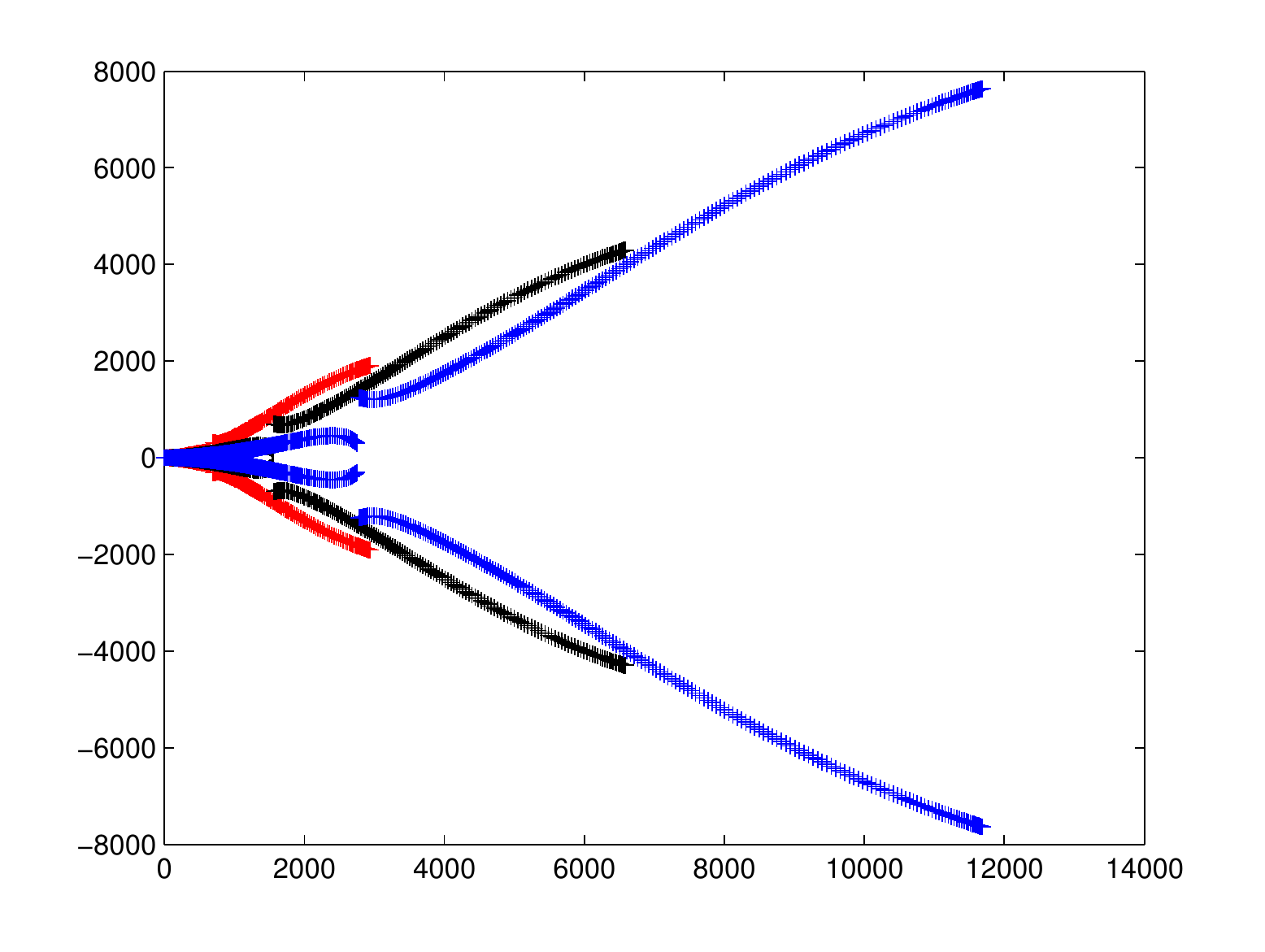} \hspace{-5mm}
\includegraphics[scale=0.4]{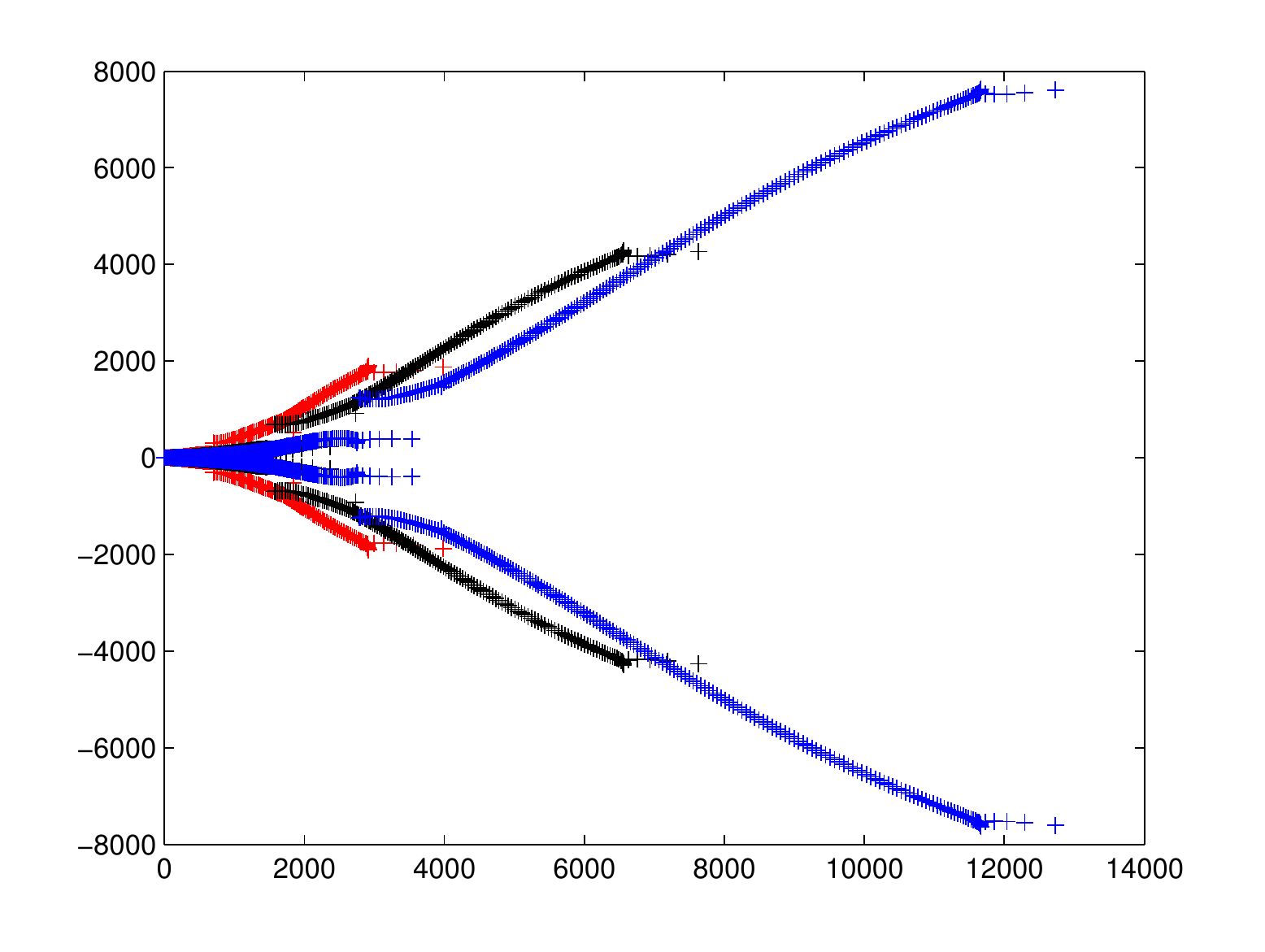}
\caption{Second order spectra of $H^{\x{har}}_6$ (left) and $H^{\x{anh}}_6$ (right) relative to $\mathcal{L}_6^h$. Here $n=100$ (red), $n=150$ (black) and $n=200$ (blue). \label{fig5_1}}
\end{figure}
\begin{figure}[H]
\centering
\includegraphics[scale=0.4]{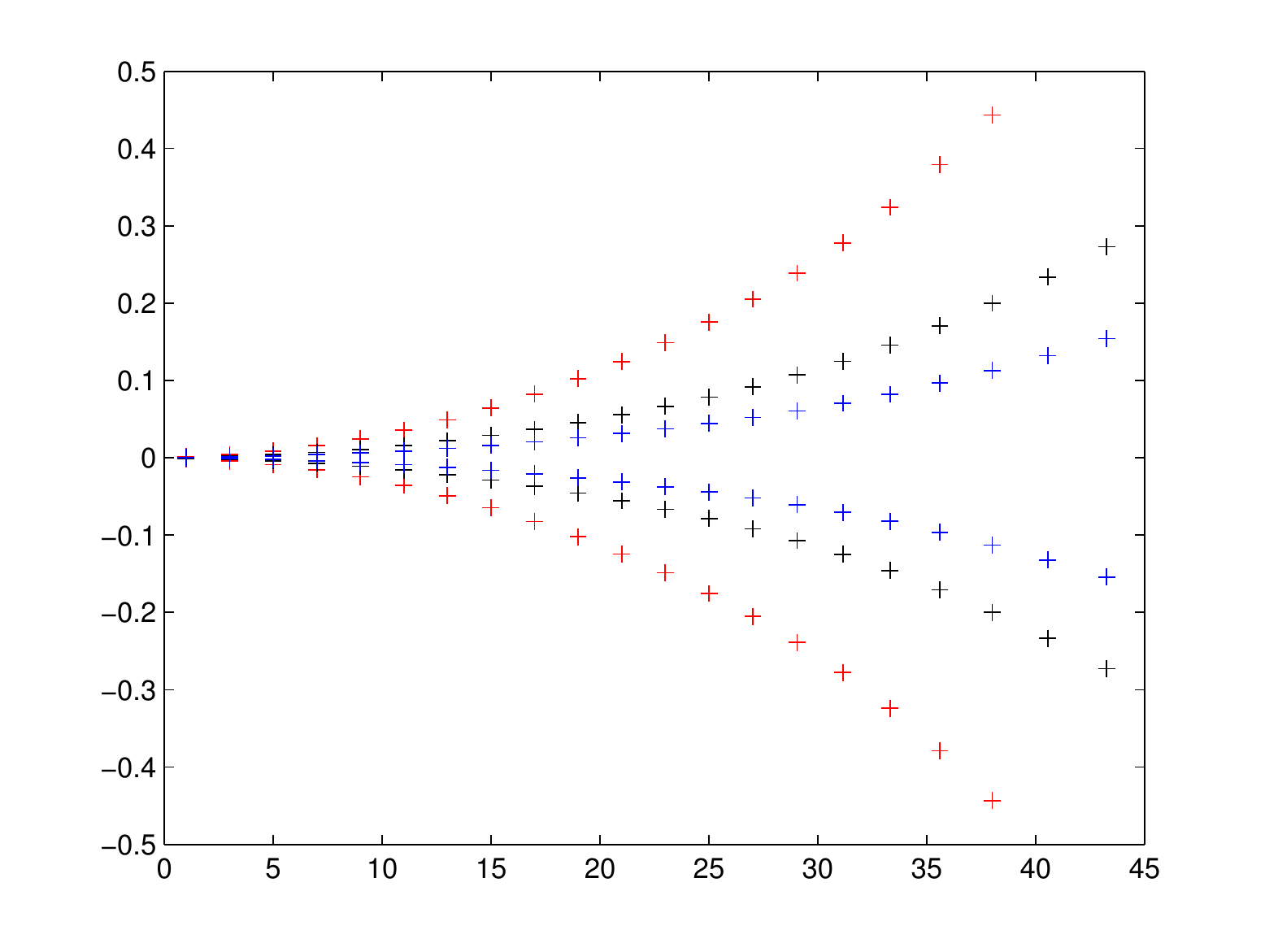} \hspace{-5mm}
\includegraphics[scale=0.4]{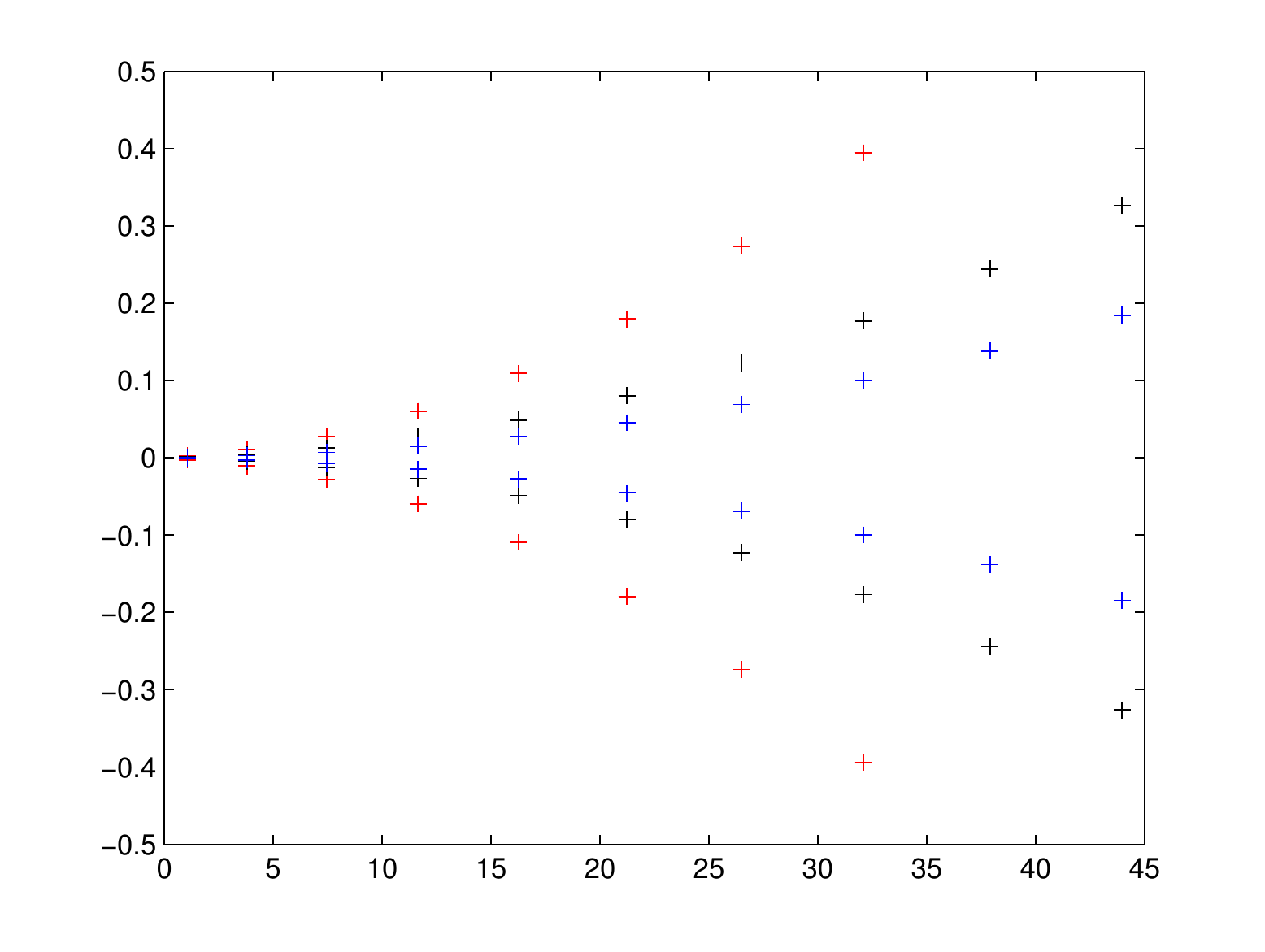}
\caption{Zoom image corresponding to  Figure~\ref{fig5_1} on a thin box near the origin. \label{fig5_2}}
\end{figure}
\begin{figure}[H]
\centering
\includegraphics[width=4.5cm]{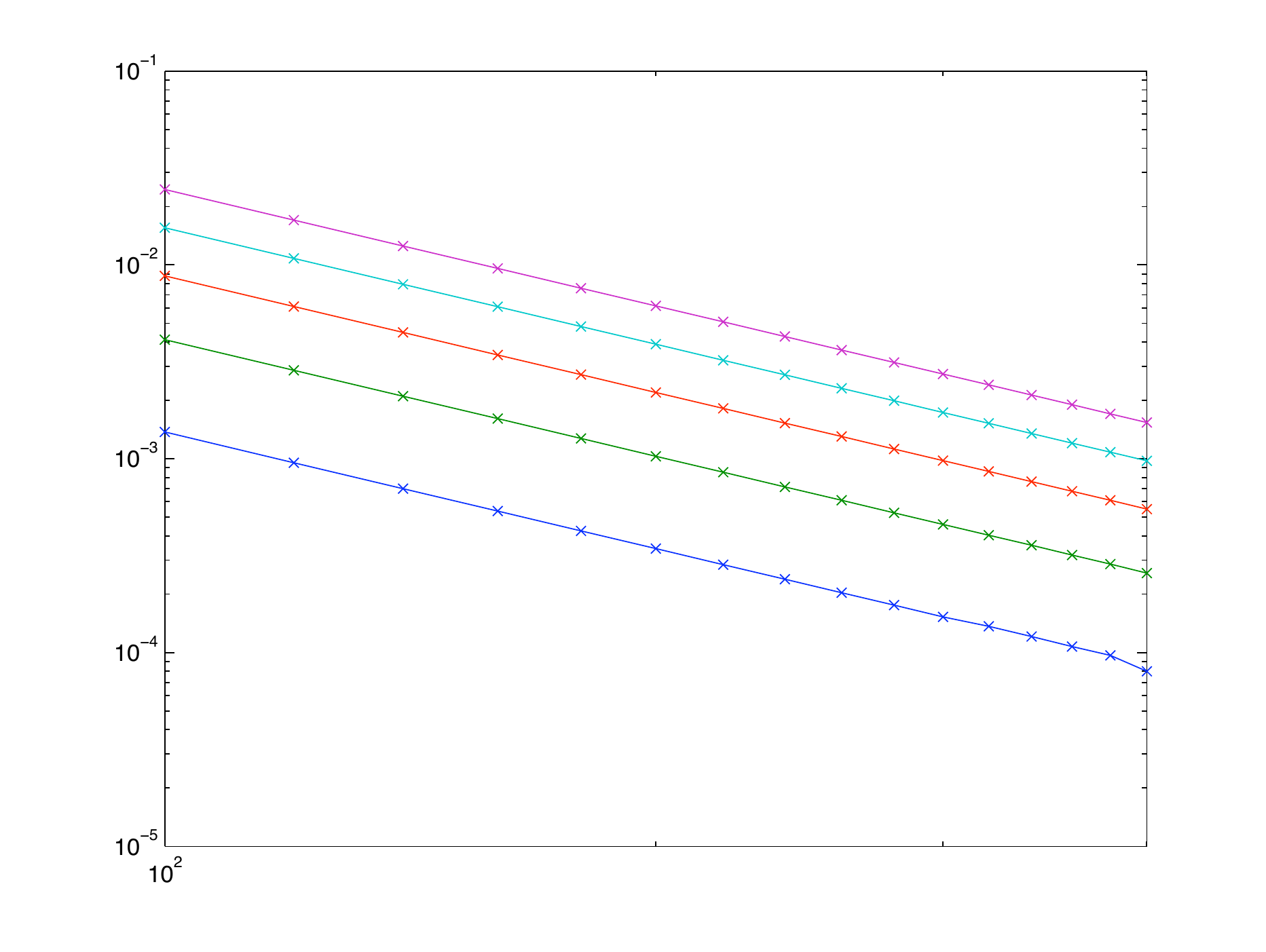} \hspace{-5mm}
\includegraphics[width=4.5cm]{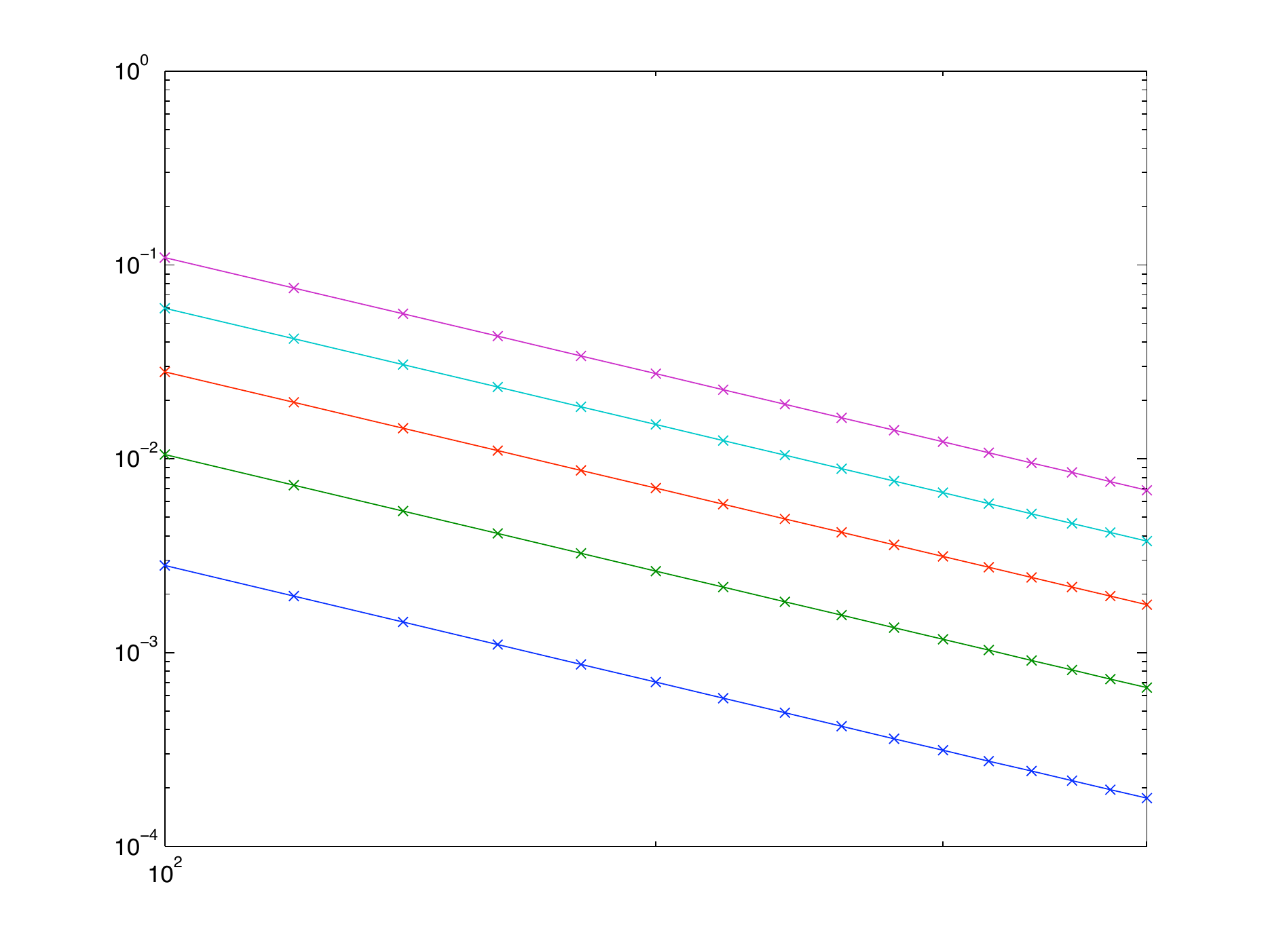} \hspace{-5mm}
\includegraphics[width=2cm]{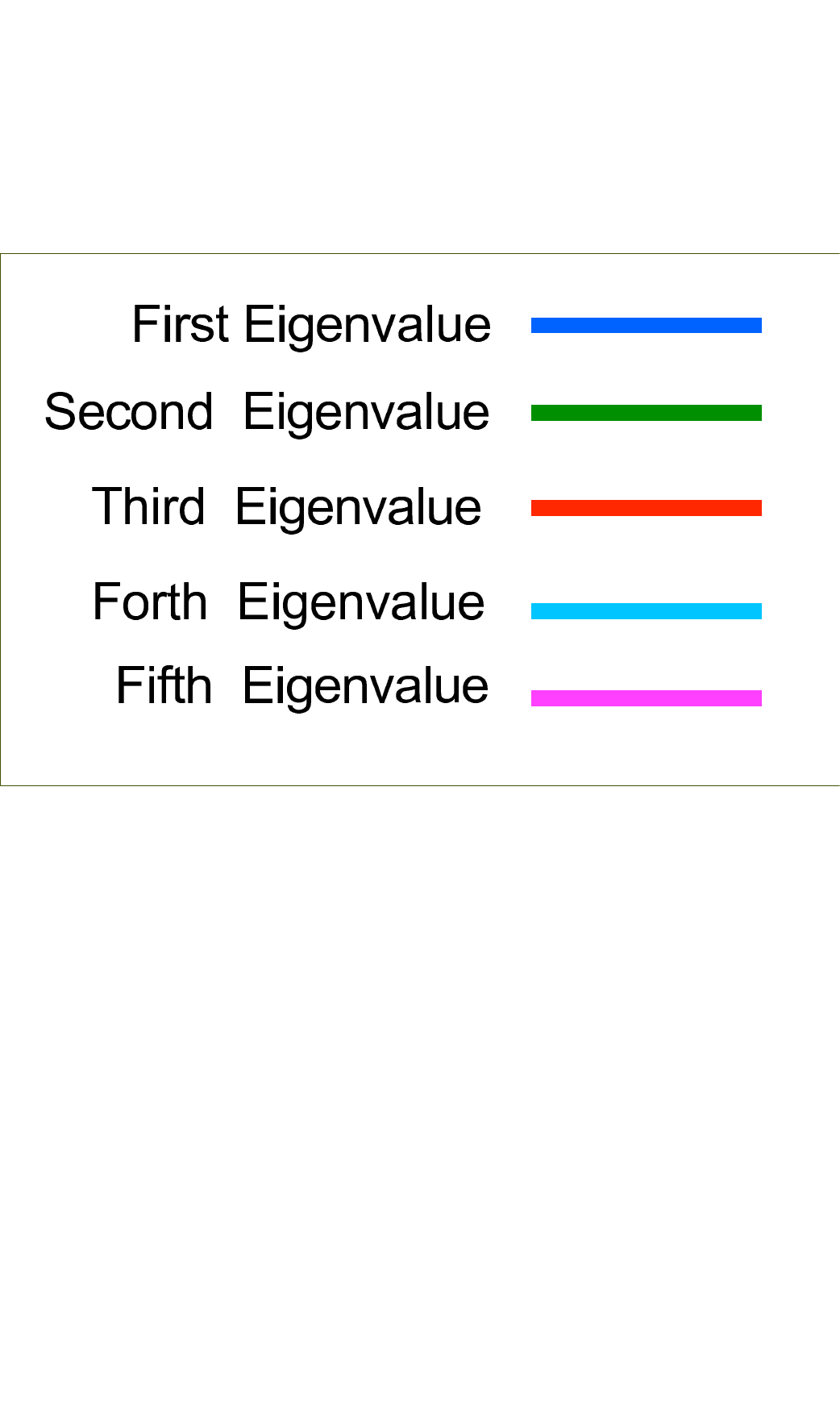}
\caption{Loglog plot of the length of the enclosure $r(j,n)$ for $H_6^{\x{har}}$ (left) and $H_6^{\x{anh}}$ (right),
 as $n$ increases. The slopes are all close to the value 2. \label{fig6}}
\end{figure}

\begin{table}[b]
\centerline{\begin{tabular}{ |c|c|c| }
  \hline
  $j$ &   harmonic &  anharmonic \\ \hline
  1 & $^{1.00009}_{0.99991}$ & $1.060^{54}_{18}$  \\ \hline
  2 & $^{3.00026}_{2.99974}$ & $3.79^{894}_{351}$ \\ \hline
  3 & $^{5.00055}_{4.99945}$ & $7.45^{747}_{393}$ \\ \hline
  4 & $^{7.00098}_{6.99902}$ & $11.64^{48}_{09}$ \\ \hline
  5 & $^{9.00154}_{8.99846}$ & $16.2^{688}_{549}$ \\ \hline
  \end{tabular}}
\caption{Approximating enclosures for first five eigenvalues of $H_6^{\x{har}}$ and $H_6^{\x{anh}}$  with $n=400$.
Here $\lambda_{\mathrm{low}}$ is the lower bound of the segment enclosing $\lambda$ and $\lambda^{\mathrm{up}}$ the upper bound. \label{table6}}
\end{table}

In Table \ref{table6} we show approximation of the first five eigenvalues of $H_6^{\x{har}}$ and $H_6^{\x{anh}}$ with $n=400$.  According to Theorem~\ref{thmshar} the numbers shown are certified upper and lower bounds for these eigenvalues.  

From the conclusion of Corollary~\ref{cor62}, the length of the intervals of enclosure for each one of these individual eigenvalue decreases at a rate proportional to $h^1$  as $h\to 0$. Note that this rate could not be verified in the past directly from the results reported in \cite[Theorem~3.4]{2010Boulton}, because $\mathcal{L}_L^h\not\subset \dom((H_L)^2)$. In Figure~\ref{fig6} we show plots in loglog scale of the number of nodes $n$ versus the exact residual 
\[
r(j,n)=(\lambda_j)^{\x{up}}-(\lambda_j)_{\x{low}}.
\]
 The slopes of the graphs are always close to the value $2$. Thus the actual rate of decrease of this residual seems to be proportional to $h^2$  as $h\to 0$. 

When the $n$ reaches a threshold $N_j$, the residual $r(j,n)$ stops decreasing. For $n>N_j$, the behaviour of  $r(j,n)$ becomes erratic. This is a consequence of rounding error taking over in the calculation of the conjugate pairs in the second order spectra. See Figure~\ref{fig7}.  These thresholds depend on the individual eigenvalues. In Table~\ref{table7} we show a heuristic prediction of the value of $N_j$ alongside the corresponding enclosure for $j=1,\ldots,5$.
 
\begin{figure}[H]
\centering
\includegraphics[width=4.5cm]{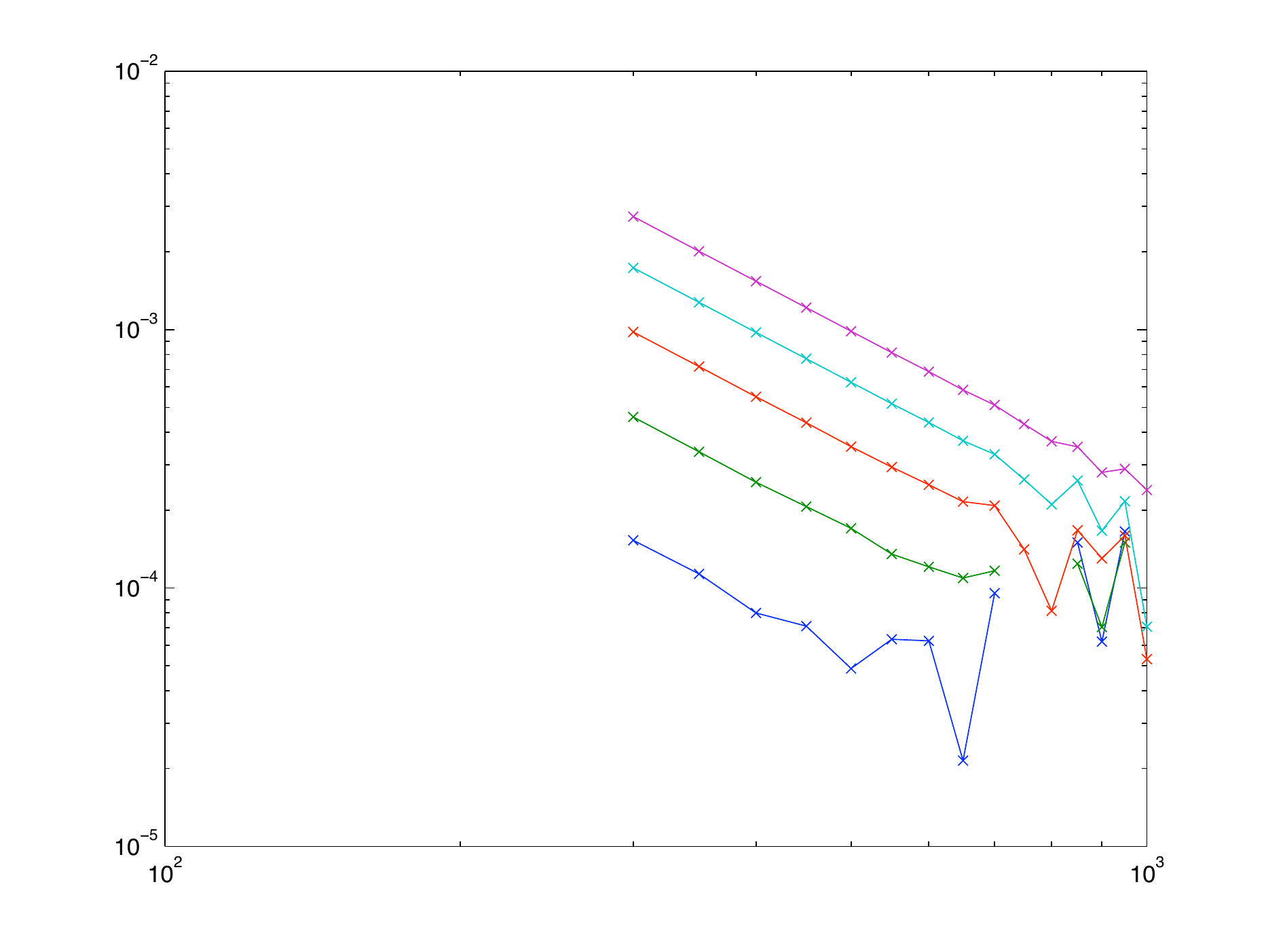}
\includegraphics[width=4.5cm]{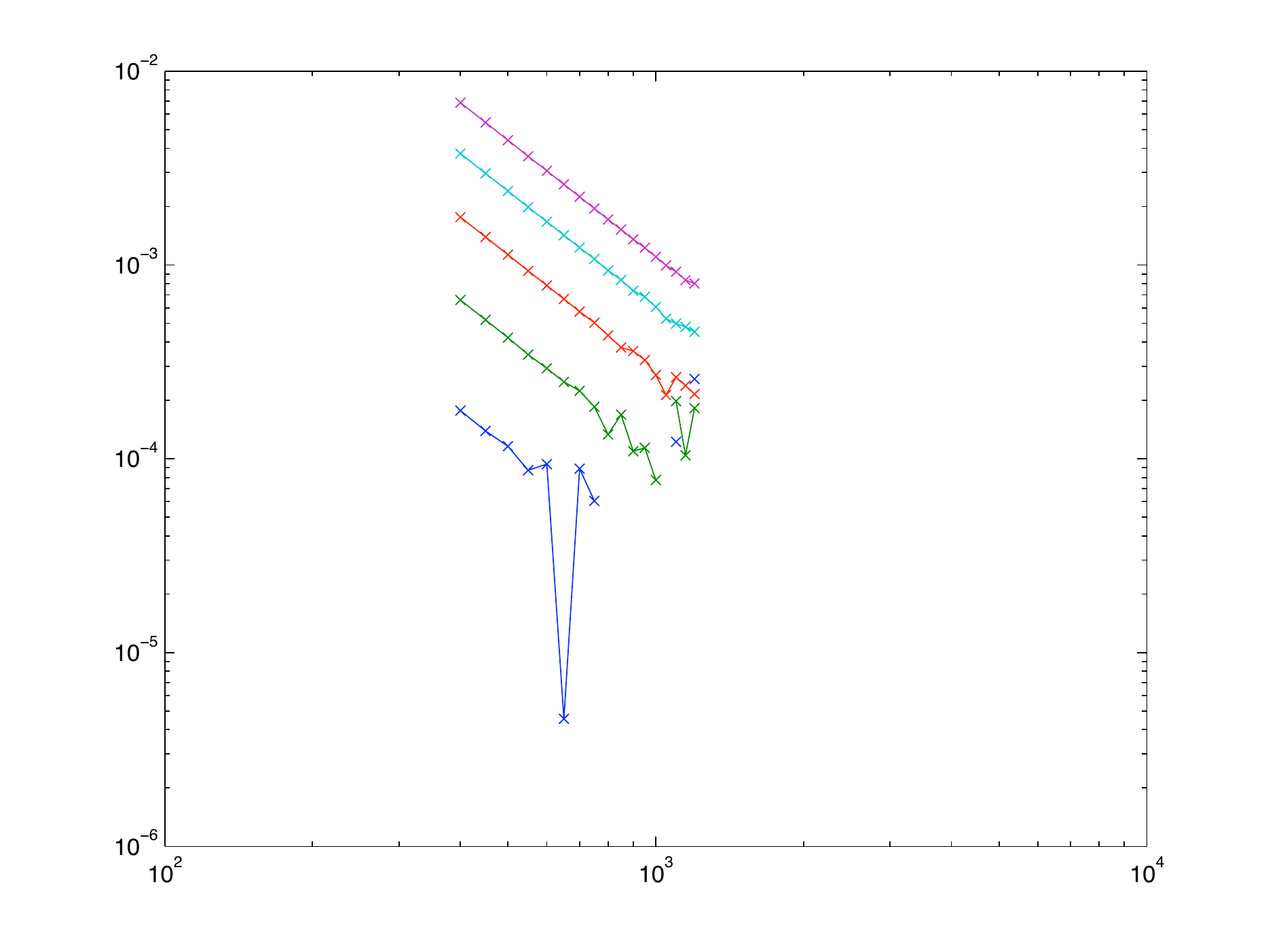}
\includegraphics[width=2cm]{legend7}
\caption{Loglog plot of the length of the enclosure $r(j,n)$ for $H_6^{\x{har}}$ (left) and $H_6^{\x{anh}}$ (right),
 as $n$ becomes very large and reaches $N_j$.    \label{fig7}}
\end{figure}
\begin{table}[H]
\centerline{\begin{tabular}{ |c|c|c||c|c| }
  \hline
  $j$ & $N_j$ &  harmonic & $N_j$ &  anharmonic \\ \hline
  1 &450 &$^{1.00008}_{0.99992}$ &550 & $1.060^{45}_{27}$  \\ \hline
  2 &550 &$^{3.00014}_{2.99986}$ &700 & $3.799^{90}_{44}$ \\ \hline
  3 &650 &$^{5.00022}_{4.99978}$ &850 & $7.45^{608}_{532}$ \\ \hline
  4 &700 &$^{7.00033}_{6.99967}$ &1050 &$11.64^{53}_{42}$  \\ \hline
  5 &800 &$^{9.00037}_{8.99963}$ &1150 & $16.26^{27}_{09}$ \\ \hline
  \end{tabular}}
\caption{Prediction of $N_j$ alongside with the corresponding enclosure.\label{table7}}
\end{table}


\section*{Appendix}

The numerical experiments shown in Section~\ref{numerical} are performed on trial spaces $\mathcal{L}_L^h$ defined as in \eqref{fespaces} and generated by Hermite elements of order $r=3$. The associated basis functions over two contiguous segments in the mesh are explicitly given by
\[ \pp_j(x) = \left\{ \begin{array}{clc}
\frac{-(x-x_{j-1})^2(x_{j-1}-3x_j+2x)}{(x_j-x_{j-1})^3} & \mbox{for} & x_{j-1}\leq x\leq x_j\\
\frac{(x-x_{j+1})^2(x_{j+1}-3x_j+2x)}{(x_{j+1}-x_j)^3} & \mbox{for}  & x_j\leq x\leq x_{j+1}\\
0 & \mbox{otherwise}
\end{array}\right.\]
and
\[ \qq_j(x) = \left\{ \begin{array}{clc}
\frac{(x-x_j)(x-x_{j-1})^2}{(x_j-x_{j-1})^2} & \mbox {for} & x_{j-1}\leq x\leq x_j\\
\frac{(x-x_j)(x-x_{j+1})^2}{(x_{j+1}-x_j)^2} & \mbox {for}  & x_j\leq x\leq x_{j+1}\\
0 & \mbox{otherwise}
\end{array}\right.\]
where $h=\frac{2L}{n}$, $-L=x_0$ and $x_n=L$. 
The matrices $\mathbf{A}_0$, $\mathbf{A}_1$ and $\mathbf{A}_2$, are banded matrices.  Set
\[
      \mathbf{B}^\ell_{jk}=\begin{bmatrix}
                       \mathfrak{a}^\ell(\pp_j,\pp_k) & \mathfrak{a}^\ell(\pp_j,\qq_k) \\
                       \mathfrak{a}^\ell(\qq_j,\pp_k) & \mathfrak{a}^\ell(\qq_j,\qq_k)
                      \end{bmatrix}
\qquad
\text{and} \qquad
\mathbf{B}^\ell=[  \mathbf{B}^\ell_{jk} ]_{jk=1}^n.
\]
Then
\[
    \mathbf{A}_\ell=\left[\begin{array}{c|c|c}
          \mathfrak{a}^\ell(\qq_0,\qq_0) & \cdots &  \mathfrak{a}^\ell(\qq_0,\qq_{n+1}) \\
          \hline & &\\
          \vdots & \qquad\scalebox{2}{$\mathbf{B}^\ell$} \qquad & \vdots \\ & & \\ \hline
    \mathfrak{a}^\ell(\qq_{n+1},\qq_0) & \cdots &  \mathfrak{a}^\ell(\qq_{n+1},\qq_{n+1})
    \end{array} \right].
\]
The entries of $\mathbf{A}_\ell$ are zero for all $|j-k|>1$. The remaining entries can be found explicitly from the tables below.

\begin{table}[H]
\center
\begin{tabular}{||cc||cc|cl|| }
\hline
 $b_j$ & $b_k$    & $\langle b_j,b_k \rangle$  &  $\langle b_j{'},b_k{'} \rangle$ &  $\langle b_j{''},b_k{''} \rangle$ & $\langle b_j,x^2b_k \rangle$ \\ \hline 
$\qq_0$ & $\qq_0$ & $\frac{h^3}{105}$ & $\frac{2h}{15}$ & $\frac{4}{h}$
& $\frac{h^5}{630}+\frac{h^4x_0}{140}+\frac{h^3x^2_0}{105}$  \\ \hline
$\qq_{n+1}$ & $\qq_{n+1}$ & $\frac{h^3}{105}$ & $\frac{2h}{15}$ & $\frac{4}{h}$
& $\frac{h^5}{630}-\frac{h^4x_n}{140}+\frac{h^3x^2_n}{105}$\\ \hline
$\qq_0$&$\pp_1$ & $\frac{13h^2}{420}$ &  $\frac{-1}{10}$  & $\frac{-6}{h^2}$
& $\frac{5h^4}{504}+\frac{13h^2x^2_0}{420}+\frac{h^3x_0}{30}$ \\ \hline 
$\qq_0$&$\qq_1$ & $\frac{-3h^3}{420}$ & $\frac{-h}{30}$&  $\frac{2}{h}$
& $\frac{-h^5}{504}-\frac{h^4x_0}{140}-\frac{h^3x^2_0}{140}$ \\ \hline 
$\qq_{n+1}$ & $\pp_{n}$ & $\frac{-13h^2}{420}$  &  $\frac{1}{10}$ & $\frac{6}{h^2}$ 
&  $\frac{-5h^4}{504}+\frac{h^3x_n}{30}-\frac{13h^2x^2_n}{420}$ \\ \hline 
$\qq_{n+1}$ & $\qq_{n}$ & $\frac{-3h^3}{420}$ & $\frac{-h}{30}$ & $\frac{2}{h}$ 
& $\frac{-h^5}{504}+\frac{h^4x_n}{140}-\frac{h^3x^2_n}{140}$\\ \hline 
$\pp_j$&$\pp_j$  & $\frac{52h}{70}$  &  $\frac{12h}{5}$& $\frac{24}{h^3}$ 
&  $\frac{19h^3}{315}+\frac{26hx^2_j}{35}$\\ \hline
$\qq_j$ & $\qq_j$ &  $\frac{8h^3}{420}$  &   $\frac{8h}{30}$  & $\frac{8}{h}$
& $\frac{2h^3x^2_j}{105}+\frac{h^5}{315}$\\ \hline
$\qq_j$ & $\pp_j$ & $0$   & $0$  & $0$ 
 & $\frac{h^3x_j}{15}$\\ \hline
$\pp_j$ & $\pp_{j\pm 1}$ & $\frac{9h}{70}$ & $\frac{-6h}{5}$ & $\frac{-12}{h^3}$ 
 & $\frac{23h^3}{630}-81hx_{j+1}+81x^2_{j+1}$\\ \hline
$\qq_j$ & $\qq_{j\pm 1}$& $\frac{-3h^3}{420}$&  $\frac{-h}{30}$ & $\frac{2}{h}$
&  $\frac{-h^5}{504}+\frac{h^4x_{j+1}}{140}-\frac{h^3x^2_{j+1}}{140}$ \\ \hline
$\pp_j$ & $\qq_{j+1}$ & $\frac{13h^2}{420}$ &  $\frac{1}{10}$  & $\frac{6}{h^2}$  
&  $\frac{19h^4}{2520}-72hx_{j+1}+78x^2_{j+1}$ \\ \hline
$\pp_j$ & $\qq_{j-1}$ & $\frac{-13h^2}{420}$& $\frac{-1}{10}$  & $\frac{-6}{h^2}$
& $\frac{-25h^4}{2520}-84hx_{j+1}+78x^2_{j+1}$ \\ \hline 
 \end{tabular}
\end{table}

\newpage

\begin{table}[H]
\hspace{-1.8cm}
\begin{tabular}{ ||cc||p{6cm}|p{8cm}||  }
\hline
   $b_j$ & $b_k$         &  $ \langle b_j,x^4b_k \rangle $   & $ \langle b_j,x^8b_k \rangle $  \\ \hline 
 $\qq_0$ & $\qq_0$  &$ \frac{h^3}{6930}[3h^4+66x^4_0+99hx^3_0+66h^2x^2_0+22h^3x_0]$ 
& 
$\frac{h^3}{45045}[429x^8_0+1287hx^7_0+2002h^2x^6_0+2002h^3x^5_0+1365h^4x^4_0+637h^5x^3_0+196h^6x^2_0+36h^7x_0+3h^8]$ \\ \hline
 $\qq_{n+1}$ & $\qq_{n+1}$&$ \frac{h^3}{6930}[3h^4+66x^4_n-99hx^3_n+66h^2x^2_n-22h^3x_n]$  
& $\frac{h^3}{45045}[429x^8_n-1287hx^7_n+2002h^2x^6_n-2002h^3x^5_n+1365h^4x^4_n-637h^5x^3_n+196h^6x^2_n-36h^7x_n+3h^8]$ \\ \hline
$\qq_0$&$\pp_1$ & $ \frac{h^2}{27720}[119h^4+858x^4_0+1848hx^3_0+1650h^2x^2_0+704h^3x_0]$ 
 &
$\frac{h^2}{180180}[5577x^8_0+24024hx^7_0+50050h^2x^6_0+64064h^3x^5_0+54145h^4x^4_0+30576h^5x^3_0+11172h^6x^2_0+2400h^7x_0+231h^8]$ 
\\ \hline 
$\qq_0$&$\qq_1$ &$\frac{-h^7}{1320}-\frac{h^3x^4_0}{140}-\frac{h^4x^3_0}{70}-\frac{h^5x^2_0}{84}-\frac{h^6x_0}{210}$ 
&
$\frac{-h^3}{180180}[1287x^8_0+5148hx^7_0+10010h^2x^6_0+12012h^3x^5_0+9555h^4x^4_0+5096h^5x^3_0+1764h^6x^2_0+360h^7x_0+33h^8]$
\\ \hline 
$\qq_{n+1}$ & $\pp_{n}$ &$ \frac{-h^2}{27720} [119h^4+858x^4_n+1848hx^3_n+1650h^2x^2_n+704h^3x_n]$ 
&
$\frac{-h^2}{180180}[5577x^8_n-24024hx^7_n+50050h^2x^6_n-64064h^3x^5_n+54145h^4x^4_n-30576h^5x^3_n+11172h^6x^2_n-2400h^7x_n+231h^8]$
\\ \hline 
$\qq_{n+1}$ & $\qq_{n}$ & $\frac{-h^7}{1320}-\frac{h^3x^4_n}{140}+\frac{h^4x^3_n}{70}-\frac{h^5x^2_n}{84}+\frac{h^6x_n}{210}$
&
$\frac{-h^3}{180180}[1287x^8_n-5148hx^7_n+10010h^2x^6_n-12012h^3x^5_n+9555h^4x^4_n-5096h^5x^3_n+1764h^6x^2_n-360h^7x_n+33h^8]$
 \\ \hline 
$\pp_j$ & $\pp_j$ &$\frac{h^5}{77}+\frac{38h^3x^2_j}{105}+\frac{26hx^4_j}{35}$ 
& 
$\frac{10h^5x^4_j}{11}+\frac{248h^7x^2_j}{2145}+\frac{26hx^8_j}{35}+\frac{76h^3x^6_j}{45}+\frac{74h^9}{45045}$
\\ \hline
$\qq_j$ & $\pp_j$  & $\frac{h^7}{1155}+\frac{2h^3x^4_j}{105}+\frac{2h^5x^2_j}{105}$ 
&
$\frac{2h^3x^8_j}{105}+\frac{56h^9x^2_j}{6435}+\frac{4h^5x^6_j}{45}+\frac{2h^7x^4_j}{33}+\frac{2h^{11}}{15015}$ 
\\ \hline
$\qq_j$ & $\pp_{j}$  &$\frac{8h^5x_j}{315}+\frac{2h^3x^3_j}{15}$ 
&
$\frac{256h^9x_j}{45045}+\frac{52h^7x^3_j}{495}+\frac{4h^3x^7_j}{15}+\frac{16h^5x^5_j}{45}$
\\ \hline
$\pp_j$ & $\pp_{j\pm 1}$ &$\frac{h}{4620}[69h^4-418h^3x_{j+1}+1012h^2x^2_{j+1}-1188hx^3_{j+1}+594x^4_{j+1}]$ 
& 
$\frac{h}{90090}[11583x^8_{j+1}-46332hx^7_{j+1}+92092h^2x^6_{j+1}-114114h^3x^5_{j+1}+94185h^4x^4_{j+1}-52234h^5x^3_{j+1}+18816h^6x^2_{j+1}-3996h^7x_{j+1}+381h^8]$
\\ \hline
$\qq_j$ & $\qq_{j\pm 1}$ & $\frac{-h^7}{1320}+\frac{h^6x_{j+1}}{210}-\frac{h^5x^2_{j+1}}{84}+\frac{h^4x^3_{j+1}}{70}-\frac{h^3x^4_{j+1}}{140}$ 
&
$\frac{-h^3x^8_{j+1}}{140}+\frac{h^4x^7_{j+1}}{35}-\frac{h^5x^6_{j+1}}{18}+\frac{h^6x^5_{j+1}}{15}-\frac{7h^7x^4_{j+1}}{130}+\frac{2h^{10}x_{j+1}}{1001}
-\frac{7h^9x^2_{j+1}}{715}+\frac{14h^8x^3_{j+1}}{495}-\frac{h^{11}}{5460}$
\\ \hline
$\pp_j$ & $\qq_{j+1}$ & $\frac{h^2}{27720}[75h^4-484h^3x_{j+1}+1254h^2x^2_{j+1}-1584hx^3_{j+1}+858x^4_{j+1}]$&
$\frac{-h^2}{180180}[5577x^8_{j+1}-24024hx^7_{j+1}+50050h^2x^6_{j+1}-64064h^3x^5_{j+1}+54145h^4x^4_{j+1}
-30576h^5x^3_{j+1}+11172h^6x^2_{j+1}-2400h^7x_{j+1}+231h^8]$
\\ \hline
$\pp_j$ & $\qq_{j-1}$ &$\frac{-h^2}{27720}[119h^4-704h^3x_{j+1}+1650h^2x^2_{j+1}-1848hx^3_{j+1}+858x^4_{j+1}]$ 
&
$\frac{h^2}{180180}[5577x^8_{j+1}-20592hx^7_{j+1}+
38038 h^2 x^{6}_{j+1} -44044h^3x^5_{j+1}+34125 h^4 x^4_{j+1}-17836 h^5 x^3_{j+1}+6076 h^6 x^2_{j+1}-1224h^7x_{j+1}+111h^8]$
\\ \hline 
 \end{tabular}
\end{table}

\begin{table}[H]
\center
\begin{tabular}{ ||cc||p{5cm}|p{6cm}||  }
  \hline
  $b_j$ &  $b_k$  &  $ \langle b{''}_j,x^2b_k \rangle $ &
 $ \langle b{''}_j,x^4b_k \rangle $  \\ \hline 
$ \qq_0$ & $\qq_0 $ & $\frac{-h}{105}[h^2+7hx_0+14x^2_0]$ &
 $\frac{hx_0}{105}[-14x^3_0-14hx^2_0-6h^2x_0-h^3]$ \\ \hline 
$ \qq_{n+1}$ & $\qq_{n+1} $ 
& $\frac{-h}{105}[h^2-7hx_n+14x^2_n]$
 &
 $\frac{hx_n}{105}[14x^3_n-14hx^2_n+6h^2x_n-h^3]$ \\ \hline
$ \qq_0$ & $\pp_1 $ 
& $\frac{x^2_0}{10}+\frac{2hx_0}{5}+\frac{23h^2}{105}$
&
 $\frac{17h^4}{84}+\frac{x_0^4}{10}+\frac{4hx^3_0}{5}+\frac{46h^2x^2_0}{35}+\frac{6h^3x_0}{7}$ \\ \hline
$ \pp_1$ & $\qq_0 $ 
&  $\frac{x^2_0}{10}-\frac{h^2}{70}$
&
 $\frac{-h^4}{84}+\frac{x^4_0}{10}-\frac{3h^2x^2_0}{35}-\frac{2h^3x_0}{35}$ \\ \hline
$ \qq_0$ & $\qq_1 $ 
& $\frac{h}{210}[7x^2_0-2h^2]$
&
 $\frac{h}{420}[-5h^4+14x^4_0-24h^2x^2_0-20h^3x_0]$ \\ \hline
$ \qq_1$ & $\qq_0 $ 
& $\frac{h}{210}[5h^2+14hx_0+7x^2_0]$
&
 $\frac{h}{420}[5h^4+28h^3x_0+60h^2x^2_0+56hx^3_0+14x^4_0]$ \\ \hline
$ \pp_{n}$ & $\qq_{n+1} $
&  $\frac{-x^2_n}{10}+\frac{h^2}{70}$
 &
$\frac{h^4}{84}-\frac{x^4_n}{10}+\frac{3h^2x^2_n}{35}-\frac{2h^3x_n}{35}$ \\ \hline 
$ \qq_{n+1}$ & $\pp_{n} $
&  $\frac{-x^2_n}{10}+\frac{2hx_n}{5}-\frac{23h^2}{105}$
 &
  $\frac{-17h^4}{84}-\frac{x^4_n}{10}+\frac{4hx^3_n}{5}-\frac{46h^2x^2_n}{35}+\frac{6h^3x_n}{7}$ \\ \hline 
$ \qq_{n}$ & $\qq_{n+1} $ 
& $\frac{h}{210}[5h^2-14hx_n+7x^2_n]$
&
 $\frac{h}{420}[5h^4-28h^3x_n+60h^2x^2_n-56hx^3_n+14x^4_n]$ \\ \hline
$ \qq_{n+1}$ & $\qq_{n} $  
&  $\frac{h}{210}[7x^2_n-2h^2]$
&
 $\frac{h}{420}[-5h^4+14x^4_n-24h^2x^2_n+20h^3x_n]$ \\ \hline
$ \pp_j$ & $\pp_j  $  
& $\frac{-2}{35h}[-h^2+42x^2_j]$ 
&
 $\frac{-4}{105h}[-2h^4-9h^2x^2_j+63x^4_j]$ \\ \hline
$ \qq_j$ & $\qq_j  $  
 & $\frac{-4}{15}[hx^2_j-\frac{2h^3}{105}]$
&
 $\frac{-4h^3x^2_j}{35}-\frac{4hx^4_j}{15}$ \\ \hline
$ \pp_j$ & $\qq_j  $  
& $0$
&
 $\frac{4h^3x^4_j}{35}$ \\ \hline
$ \qq_j$ & $\pp_{j}  $  
& $\frac{-4hx_j}{5}$
&
 $\frac{-4h^3x_j}{35}+\frac{8hx^3_j}{5}$ \\ \hline 
$ \pp_j$ & $\pp_{j+1}  $ 
& $\frac{1}{35h}[-h^2-7hx_{j+1}+42x^2_{j+1}]$
&
  $\frac{2}{105h}[-2h^4+9h^3x_{j+1}-9h^2x^2_{j+1}-21hx^3_{j+1}+63x^4_{j+1}]$ \\ \hline 
$ \pp_j$ & $\pp_{j-1}  $ 
& $\frac{1}{35h}[34h^2-77hx_{j+1}+42x^2_{j+1}]$
&
 $\frac{2}{105h}[40h^4-180h^3x_{j+1}+306h^2x^2_{j+1}-231hx^3_{j+1}+63x^4_{j+1}]$ \\ \hline 
$ \qq_j$ & $\qq_{j+1}  $ 
& $\frac{h}{210}[5h^2-14hx_{j+1}+7x^2_{j+1}]$ 
&
 $\frac{h}{420}[5h^4-28h^3x_{j+1}+60h^2x^2_{j+1}-56hx^3_{j+1}+14x^4_{j+1}]$ \\ \hline  
$ \qq_j$ & $\qq_{j-1}  $
& $\frac{h}{210}[-2h^2+7x^2_{j+1}]$
&
 $\frac{h}{420}[-5h^4+20h^3x_{j+1}-24h^2x^2_{j+1}+14x^4_{j+1}]$ \\ \hline 
$ \pp_j$ & $\qq_{j+1}  $ 
& $\frac{h^2}{70}-\frac{x^2_{j+1}}{10}$
&
 $\frac{-h^4}{84}-\frac{2h^3x_{j+1}}{35}+\frac{3h^2x^2_{j+1}}{35}-\frac{x^4_{j+1}}{10}$ \\ \hline 
$ \qq_j$ & $\pp_{j-1}  $ 
& $\frac{-23h^2}{105}+\frac{2hx_{j+1}}{5}-\frac{x^2_{j+1}}{10}$
& $\frac{-17h^4}{84}+\frac{6h^3x_{j+1}}{7}-\frac{46h^2x^2_{j+1}}{35}+\frac{4hx^3_{j+1}}{5}-\frac{x^4_{j+1}}{10}$ \\ \hline
$ \qq_j$ & $\pp_{j+1}  $ 
 & $\frac{-17h^2}{210}+\frac{hx_{j+1}}{5}+\frac{x^2_{j+1}}{10}$
& $\frac{-17h^4}{420}+\frac{8h^3x_{j+1}}{35}-\frac{17h^2x^2_{j+1}}{35}+\frac{2hx^3_{j+1}}{5}+\frac{x^4_{j+1}}{10}$ \\ \hline
$ \pp_j$ & $\qq_{j-1}  $ 
& $\frac{3h^2}{35}-\frac{hx_{j+1}}{5}+\frac{x^2_{j+1}}{10}$
&  $\frac{-5h^4}{84}-\frac{2h^3x_{j+1}}{7}+\frac{18h^2x^2_{j+1}}{35}-\frac{2hx^3_{j+1}}{5}+\frac{x^4_{j+1}}{10}$ \\ \hline
\end{tabular}
\end{table}

\section*{Acknowledgements}
Financial support was provided by the Engineering and Physical Sciences Research Council (grant number EP/I00761X/1) and King Abdulaziz University. We kindly thank M~Alsaeed and N~Boussa{\"\i}d  for their suggestions during the preparation of this paper.

\bibliographystyle{siam}

\begin{thebibliography}{10}

\bibitem{200666Boulton}
{\sc L.~Boulton}, {\em Limiting set of second order spectra}, Math. Comp., 75
  (2006), pp.~1367--1382.

\bibitem{20077Boulton}
\leavevmode\vrule height 2pt depth -1.6pt width 23pt, {\em Non-variational
  approximation of discrete eigenvalues of self-adjoint operators}, IMA J.
  Numer. Anal., 27 (2007), pp.~102--121.

\bibitem{2008Boulton}
{\sc L.~Boulton and N.~Boussa{\"\i}d}, {\em Non-variational computation of the
  eigenstates of {D}irac operators with radially symmetric potentials}, LMS J.
  Comput. Math., 13 (2010), pp.~10--32.

\bibitem{2012Boulton}
{\sc L.~Boulton, N.~Boussa{\"{\i}}d, and M.~Lewin}, {\em Generalised {W}eyl
  theorems and spectral pollution in the {G}alerkin method}, J. Spectr. Theory,
  2 (2012), pp.~329--354.

\bibitem{2006Boulton}
{\sc L.~Boulton and M.~Levitin}, {\em On approximation of the eigenvalues of
  perturbed periodic {S}chr\"odinger operators}, J. Phys. A, 40 (2007),
  pp.~9319--9329.

\bibitem{20066Boulton}
{\sc L.~Boulton and M.~Strauss}, {\em Stability of quadratic projection
  methods}, Oper. Matrices, 1 (2007), pp.~217--233.

\bibitem{2010Boulton}
\leavevmode\vrule height 2pt depth -1.6pt width 23pt, {\em On the convergence
  of second-order spectra and multiplicity}, Proc. R. Soc. Lond. Ser. A Math.
  Phys. Eng. Sci., 467 (2011), pp.~264--284.

\bibitem{1983Chatelin}
{\sc F.~Chatelin}, {\em Spectral approximation of linear operators}, Computer
  Science and Applied Mathematics, Academic Press, New York, 1983.

\bibitem{1978Ciarlet}
{\sc P.~Ciarlet}, {\em The finite element method for elliptic problems},
  North-Holland, Amsterdam, 1978.

\bibitem{1995Davies}
{\sc E.~B. Davies}, {\em Spectral theory and differential operators}, vol.~42
  of Cambridge Studies in Advanced Mathematics, Cambridge University Press,
  Cambridge, 1995.

\bibitem{1998Davies}
\leavevmode\vrule height 2pt depth -1.6pt width 23pt, {\em Spectral enclosures
  and complex resonances for general self-adjoint operators}, LMS J. Comput.
  Math., 1 (1998), pp.~42--74.

\bibitem{1949Kato}
{\sc T.~Kato}, {\em On the upper and lower bounds of eigenvalues}, J. Phys.
  Soc. Japan, 4 (1949), pp.~334--339.

\bibitem{kato}
{\sc T.~Kat{\=o}}, {\em Perturbation theory for linear operators}, vol.~132,
  Springer Verlag, 1995.

\bibitem{Levitin}
{\sc M.~Levitin and E.~Shargorodsky}, {\em Spectral pollution and second-order
  relative spectra for self-adjoint operators}, IMA J. Numer. Anal., 24 (2004),
  pp.~393--416.

\bibitem{1980Barry}
{\sc M.~Reed and B.~Simon}, {\em Methods of modern mathematical physics. {IV}.
  {A}nalysis of operators}, Academic Press, New York, 1978.

\bibitem{Shargorodsky}
{\sc E.~Shargorodsky}, {\em Geometry of higher order relative spectra and
  projection methods}, J. Operator Theory, 44 (2000), pp.~43--62.

\bibitem{1982Barry}
{\sc B.~Simon}, {\em Schr\"odinger semigroups}, Bull. Amer. Math. Soc. (N.S.),
  7 (1982), pp.~447--526.

\bibitem{2010strauss}
{\sc M.~Strauss}, {\em Quadratic projection methods for approximating the
  spectrum of self-adjoint operators}, IMA J. Numer. Anal., 31 (2011),
  pp.~40--60.

\bibitem{1974Weinberger}
{\sc H.~F. Weinberger}, {\em Variational methods for eigenvalue approximation},
  Society for Industrial and Applied Mathematics, Philadelphia, 1974.

\end{thebibliography}

\vspace{1cm}

\begin{minipage}{4cm} 
Lyonell Boulton$^1$ \\ \texttt{L.Boulton@hw.ac.uk}
\end{minipage}
\qquad \& \quad \qquad \begin{minipage}{4cm}
 Aatef Hobiny$^{1,2}$ \\ \texttt{ahobany@kau.edu.sa}
\end{minipage}\\ \\
\noindent \begin{minipage}{12cm} 

\medskip

\centerline{$^1$Department of Mathematics and 
Maxwell Institute for Mathematical Sciences}
\centerline{Heriot-Watt University, Edinburgh, EH14 4AS, UK}

\medskip

\centerline{$^2$ Department of Mathematics, Faculty of Science, King Abdulaziz University }
\centerline{P.O. Box. 80203, Jeddah 21589, Saudi Arabia}
\end{minipage}

 \end{document}